\documentclass[11pt,leqno, a4paper]{amsart}
\usepackage{amsmath, amssymb, amscd, amsthm, mathrsfs, hyperref}
\usepackage{amscd}
\usepackage{array}
\usepackage{amssymb, ifpdf, xspace}
\usepackage{enumitem, color}
\usepackage{fancyhdr}
\usepackage{graphicx}
\usepackage{commath}
\usepackage{mathtools}
\usepackage{bbm}
\usepackage{caption}

\let\oldtocsection=\tocsection
 
\let\oldtocsubsection=\tocsubsection
 
\let\oldtocsubsubsection=\tocsubsubsection
 
\renewcommand{\tocsection}[2]{\hspace{0em} \vspace{-1.2em} \oldtocsection{#1}{#2}}
\renewcommand{\tocsubsection}[2]{\hspace{1em} \vspace*{-1.2em} \oldtocsubsection{#1}{#2}}
\renewcommand{\tocsubsubsection}[2]{\hspace{2em} \vspace*{-1.2em} \oldtocsubsubsection{#1}{#2}}

\newtheorem{theorem}{Theorem}[section]

\newtheorem{lemma}[theorem]{Lemma}

\newtheorem{corollary}[theorem]{Corollary}
\newtheorem{proposition}[theorem]{Proposition}

\theoremstyle{definition}
\newtheorem{definition}[theorem]{Definition}

\newtheorem{example}[theorem]{Example}

\theoremstyle{remark}
\newtheorem{remark}[theorem]{Remark}

\numberwithin{equation}{section}

\newcommand{\cH}{\mathcal{H}}

\newcommand{\cM}{\mathcal{M}}
\newcommand{\cO}{\mathcal{O}}

\newcommand{\cG}{\mathcal{G}}

\newcommand{\eye}{\mathbbm{1}}

\newcommand{\N}{\mathbb{N}}

\newcommand{\R}{\mathbb{R}}

\newcommand{\cK}{\mathcal{K}}

\newcommand{\spt}{{\rm{spt}}}

\newcommand{\diam}{\operatorname{diam}}

\newcommand{\dist}{\operatorname{dist}}

\newcommand{\Span}{\operatorname{span}}

\newcommand{\aff}{\mathrm{aff}}
\renewcommand{\tilde}{\widetilde}
\renewcommand{\hat}{\widehat}

\newcommand{\curv}{\operatorname{curv}}

\newcommand{\obeta}{\hat{\beta}}

\newcommand{\defeq}{\vcentcolon=}

\renewcommand{\dif}{d}

\def\XXint#1#2#3{{\setbox0=\hbox{$#1{#2#3}{\int}$ }
\vcenter{\hbox{$#2#3$ }}\kern-.6\wd0}}

\newcommand{\restr}{\mathbin{\vrule height 1.6ex depth 0pt width
0.13ex \vrule height 0.13ex depth 0pt width 1.3ex}}

\title{Characterizations of countably $n$-rectifiable Radon measures by higher-dimensional Menger curvatures}
\author{Max Goering}
\address{Department of Mathematics, University of Washington, Seattle, WA
98195, USA}
\email{mgoering@uw.edu}
\date{\today}
				
\thanks{The author was partially supported by NSF grant number DMS-1664867.}
\subjclass[2010]{Primary 28A75. Secondary 49Q15, 42B99}
\keywords{Measure and integration, calculus of variations and optimal control, harmonic analysis on Euclidean spaces.}

\begin{document}
\maketitle

\begin{abstract}
We provide a characterization of countably $n$-rectifiable measures in terms of $\sigma$-finiteness of the integral Menger curvature. We also prove that a finiteness condition on pointwise Menger curvature can characterize rectifiability of Radon measures. Motivated by the partial converse of Meurer's work by Kolasi{\'n}ski we prove that under suitable density assumptions there is a comparability between pointwise-Menger curvature and the sum over scales of the centered $\beta$-numbers at a point. 
%

\end{abstract}

\vspace{-.4in}
\tableofcontents

\section{Introduction}

In the late 1990s there was a flurry of activity relating $1$-rectifiable sets, boundedness of singular integral operators, the analytic capacity of a set, and the integral Menger curvature in the plane. In 1999 L{\'e}ger extended the results for Menger curvature to $1$-rectifiable sets in higher dimension, as well as to the codimension one case.

A decade later, Lerman and Whitehouse, and later Meurer, found higher-dimensional geometrically motivated generalizations of Menger curvature that yield results about the uniform rectifiability of measures and the rectifiability of sets respectively.

Primarily, these higher-dimensional Menger curvatures have been used to study the regularity of surfaces and knots, for instance, self-avoidance and smoothness of the normal. Herein, we use these tools to find new characterizations of rectifiable Radon measures in arbitrary dimension and codimension (see Theorems \ref{t:maincharacterization} and \ref{t:i1}). Work in progress indicates the characterization in Theorem \ref{t:i1} (4) and (5) is more likely more useful in practice than the characterization in Theorem \ref{t:maincharacterization}. We also relate the pointwise Menger curvature to the sum of the $\beta$ numbers over scales (see Theorem \ref{t:mainbeta}). As a consequence of taking tools from knot theory to answer geometric measure theory questions, we include extra details to ensure this paper is sufficiently self-contained for readers from either discipline. Despite this, when classic arguments make repeat appearances, we attempt to avoid repetition by referring the reader to the analogous argument earlier in the paper.

\subsection{(Uniform) Rectifiability and $\beta_{p}$-coefficients}

Studying rectifiable sets and measures (see Definition \ref{d:cr}) is a central topic in geometric measure theory. In his 1990 work on the Analysts' traveling salesman problem in the plane \cite{jones1990rectifiable} (later generalized to $1$-dimensional sets in $\R^{m}$ by Okikiolu in \cite{okikiolu1992characterization} and to $n$-dimensional sets in $\R^{m}$ by Pajot in \cite{pajot1997conditions}) Peter Jones introduced what are now called the Jones' $\beta$-numbers, which have dominated the landscape in quantitative techniques relating to rectifiability, analytic capacity, and singular integrals.

In the joint monograph on the topic \cite{david1993analysis} David and Semmes laid the framework to understanding the quantitative structures Jones introduced, as well as how to properly generalize these ideas to Ahlfors regular sets and measures higher dimensions. In doing so, they introduced the notion of uniform rectifiability (see Definition \ref{d:ur}).

David and Semmes gave many equivalent characterization of uniform $n$-rectifiability. One characterization is related to Jones' $\beta$-numbers the definition of which is included for completeness. For $1 \le p < \infty$ one defines
\begin{equation} \label{e:betadef}
\beta_{\mu;p}^{n}(x,r) = \inf_{L} \left( \frac{1}{r^{n}} \int_{B(x,r)} \left( \frac{ \dist(y,L)}{r} \right)^{p} \dif \mu(y) \right)^{\frac{1}{p}},
\end{equation}
where the infimum is taken over all $n$-dimensional affine subspaces $L \subset \R^{m}$. When the dimension $n$ is understood from context, the superscript is typically forgotten.

The relevant characterization of uniformly $n$-rectifiable sets was discovered in \cite{david1991singular}.
\begin{theorem}[\cite{david1991singular}] \label{t:betauniformlyrectifiable} \ \\
If $\mu$ is an $n$-Ahlfors regular measure on $\R^{m}$ then the following are equivalent. 

\noindent (1) $\mu$ is uniformly rectifiable. \\
(2) For $1 \le p < \frac{2n}{n-2}$ there exists some $c > 0$ depending on $p$ such that the $\beta_{p}$-numbers satisfy the following so called Carleson-condition.
\begin{equation} \label{e:betapcarleson}
\int_{B(x,R)} \int_{0}^{R} \beta_{\mu;p}^{n}(y,r)^{2} \frac{ \dif r}{r} \dif \mu(y) \le c R^{n} \quad \text{for all } x \in \spt \mu, ~ R > 0.
\end{equation}
\end{theorem} 
More recently, a much sought after characterization of countably $n$-rectifiable measures (a characterization for measures, similar to the characterization for sets from \cite{pajot1997conditions}) was discovered by Azzam and Tolsa in a pair of papers \cite{tolsa2015characterization} and \cite{azzam2015characterization}. In particular, the following theorem is from \cite{tolsa2015characterization}.

\begin{theorem}[\cite{tolsa2015characterization}] \label{t:tolchar}
Let $1 \le p \le 2$. If $\mu$ is a finite Borel measure on $\R^{m}$ which is countably $n$-rectifiable, then 
\begin{equation}\label{e:betaprectifiable}
\int_{0}^{\infty} \beta_{\mu;p}^{n}(x,r)^{2} \frac{ \dif r}{r} < \infty \quad \text{for } \mu~ a.e. ~ x \in \R^{m}.
\end{equation}
\end{theorem}
The converse has hypothesis on the densities of the measure (see Definition \ref{d:densities}).

\begin{theorem} [\cite{azzam2015characterization}] \label{t:azzchar}
 Let $\mu$ be a finite Borel measure in $\R^{m}$ such that $0 < \Theta^{n,*}(\mu,x) < \infty$ for $\mu$-a.e. $x \in \R^{m}$. If
\begin{equation} \label{e:beta2rectifiable}
\int_{0}^{\infty} \beta^{n}_{\mu;2}(x,r)^{2} \frac{ \dif r}{r} < \infty \quad \mu~ a.e. ~ x \in \R^{m},
\end{equation}
then $\mu$ is countably $n$-rectifiable.
\end{theorem}

Later, \cite{edelen2016quantitative} (with an alternative proof in \cite{tolsa2017rectifiability}) removed the \emph{a priori} assumption of absolute continuity with respect to the Hausdorff measure. Together, these results can be summarized in the following theorem.

\begin{theorem}[\cite{edelen2016quantitative}] \label{t:ENV} 
Let $\mu$ be a Radon measure in $\R^{m}$ with $0 < \Theta^{n,*}(\mu,x)$ and $\Theta^{n}_{*}(\mu,x) < \infty$. Then $\mu$ is countably $n$-rectifiable if and only if 
\begin{equation} \label{e:i1}
\int_{0}^{1} \beta_{\mu,2}^{n}(x,r)^{2} \frac{ \dif r}{r} < \infty \quad \text{for } \mu~ a.e. ~ x \in \R^{m}.
\end{equation}
\end{theorem}

We note that prior to \cite{edelen2016quantitative} some partial results on rectifiability of measures without the \emph{a priori} assumption of absolute continuity with respect to the Hausdorff measure were also demonstrated in \cite{badger2015multiscale}, \cite{badger2016two}, \cite{badger2017multiscale}.

In light of the characterization for uniformly rectifiable measures, one might expect that the restriction to $p =2$ in Theorem \ref{t:azzchar} is a consequence of the proof, and potentially a range of values for $p$ could be found. However, in \cite{tolsa2017rectifiability} a very flexible construction of $1$-rectifiable sets with positive and finite measure was laid out, which, by varying the parameters of this construction demonstrate that for $1 \le p < 2$ no conclusion can be drawn from $\mu$ almost everywhere finiteness of the quantity in \eqref{e:betaprectifiable}. Under the condition that $\Theta^{n,*}(\mu,x) < \infty$, H{\"o}lder's inequality guarantees that if $p > 2$ then   \eqref{e:betaprectifiable} implies  \eqref{e:beta2rectifiable}, so only the range $1 \le p < 2$ was of particular interest.

\subsection{Uniform Rectifiability and integral Menger curvatures}
In 1995 Melnikov discovered a beautiful identity connecting Menger curvature to the analytic capacity in the plane, which kicked off a great deal of research to find curvature-based techniques to classify rectifiable and uniformly rectifiable sets, and further to find relations to singular integrals (\cite{melnikov1995analytic}, \cite{melnikov1995geometric}, \cite{mattila1996cauchy}, \cite{mattila1998rectifiability}, \cite{leger1999menger}).

Unfortunately, all of the progress made applied exclusively to dimension $1$ and/or codimension $1$ sets. In fact, in \cite{farag1999riesz}, Farag showed that in the higher-(co)dimension case there does not exist an algebraic generalization of Melnikov's identity that could directly relate to Riesz kernels. 
In particular, this result does not preclude the usefulness of Menger-type curvatures whose integrand take a different form.

After a hiatus in progress on this subject, Lerman and Whitehouse showed that a geometrically motivated generalization of Menger curvature was sufficient to characterize uniformly $n$-rectifiable measures in higher codimension (in fact, their results hold for real separable Hilbert spaces). The work, completed in \cite{lerman2009high} and \cite{lerman2011high} was motivated by the original work of David and Semmes \cite{david1991singular} and \cite{david1993analysis}, but required a nontrivial and creative idea of ``geometric multipoles". They defined several versions of Menger curvatures of $(n+2)$-points in $\R^{m}$ (see \cite{lerman2009high}, \cite{lerman2011high}, and \cite{lerman2012least}). Among these curvatures, one that is well suited for our setting is
\begin{equation} \label{e:LWcurv}
\cK_{1}^{2}(x_{0}, \dots, x_{n+1}) = \frac{ \cH^{n+1}(\Delta(x_{0},\dots,x_{n+1}))^{2}}{\diam\{x_{0}, \dots, x_{n+1}\}^{(n+1)(n+2)}}
\end{equation}
where $\Delta(x_{0}, \dots, x_{n+1})$ is the simplex with corners $\{x_{0}, \dots, x_{n+1}\}$. Lerman and Whitehouse showed that control of integrals of the integrand $\cK_{1}$ has interesting geometric consequence, this is explained in more detail below.

One necessary object to understand and state the results of Lerman and Whitehouse, is a notion of the ``space of well-scaled simplices". More precisely, for some $0 < \lambda < 1$ and letting $X = (x_{0}, \dots, x_{n+1})$ denote an $(n+2)$-tuple in $\R^{m}$, define 
\begin{equation} \label{e:wlb}
W_{\lambda}(B(x,r)) = \{ X \in B(x,r)^{n+2} : \frac{ \min(X)}{\diam(X)} \ge \lambda > 0 \}.
\end{equation}
Where  
$$
\min(X) = \min_{0 \le i < j \le n+1} |x_{i} - x_{j}|.
$$

Then, one can think of $W_{\lambda}(B(x,r))$ as the space of well-scaled simplices in $B(x,r)$. We also define the quantities
\begin{equation} \label{e:LWintcurv}
c_{1}^{2}(\mu|_{B(x,r)}) = \int_{B(x,r)^{n+2}} \cK^{2}_{1}(X) \dif \mu^{n+2}(X).
\end{equation}
and
\begin{equation} \label{e:LWintcurv-holes}
c_{1}^{2}(\mu|_{B(x,r)},\lambda) = \int_{W_{\lambda}(B(x,r))} \cK_{1}^{2}(X) \dif \mu^{n+2}(X)
\end{equation} 
The measure $\mu^{n+2}$ is the measure on $(\R^{m})^{n+2}$ resulting from taking $(n+2)$-products of $\mu$ with itself.

The quantity in  \eqref{e:LWintcurv} can be thought of as the integral Menger curvature of the measure $\mu \restr B(x,r)$. The quantity in  \eqref{e:LWintcurv-holes} can be thought of as the amount of integral Menger curvature of the measure $\mu \restr B(x,r)$ which arises from ``well-scaled simplices". The main results of \cite{lerman2009high} and \cite{lerman2011high} are combined to show the following theorem, which notably holds in the more difficult setting of real-seperable Hilbert spaces.

\begin{theorem}[\cite{lerman2009high}, \cite{lerman2011high}]\label{t:LW}
If $\mu$ is an $n$-Ahlfors regular measure on a possibly infinite dimensional, real, separable Hilbert space, then the following are true.

\begin{equation*} \label{e:curvbelow}c_{1}^{2}( \mu|_{B(x,R)}) \le C_{2} \int_{B(x,R)} \int_{0}^{12R} \beta_{\mu;2}^{n}(y,r)^{2} \frac{ \dif r}{r} \dif \mu(y) 
\end{equation*}

and there exists a $\lambda \in (0,1)$ such that for all $B(x,R)$ with $2R \le \diam ( \spt \mu)$,
\begin{equation*} \label{e:curvabove}
\int_{B(x,R)} \int_{0}^{2R} \beta_{\mu;2}^{n}(y,r)^{2} \frac{ \dif r}{r} \dif \mu(y) \le C_{3} \cdot c_{1}^{2}( \mu |_{3 \cdot B(x,R)}, \lambda/2) 
\end{equation*}
Here, both $C_{2}$ and $C_{3}$ depend only on $n$ and the Ahlfors regularity constants of $\mu$.
\end{theorem}

In particular, combining the two parts of Theorem \ref{t:LW} with the observation that $c_{1}^{2}(\mu|_{B(x,R)},\lambda) \le c_{1}^{2}(\mu|_{B(x,R)})$, it follows that 
$$
\displaystyle{c_{1}^{2}(\mu|_{B(x,R)}) \le C R^{n} \iff \int_{B(x,R)} \int_{0}^{R} \beta_{\mu;2}^{n}(x,r) \frac{ \dif r}{r} \le \tilde{C} R^{n}},
$$ 
which by passing through Theorem \ref{t:betauniformlyrectifiable} leads to the following characterization.

\begin{theorem}[\cite{david1991singular}, \cite{lerman2009high}, \cite{lerman2011high}] \label{t:LWchar}
If $n \ge 2$ and $\mu$ is an $n$-Ahlfors regular measure on $\R^{m}$ then the following are equivalent: \\
\noindent 1) There exists a constant $C$ independent of $x$ and $R$ so that $c_{1}^{2}(\mu|_{B(x,R)}) \le C R^{n}$ for all $x \in \spt \mu$ and $R > 0$. \\
\noindent 2) $\mu$ is uniformly $n$-rectifiable.
\end{theorem}

\subsection{Rectifiability and pointwise Menger curvature}

In 2015, Meurer found a general class of Menger-type curvatures which satisfy a one-sided comparison to $\beta_{p}$-coefficients and proved a sufficient condition for rectifiability of a set based off these integral menger curvatures. The integrand $\cK_{1}$ of Lerman and Whitehouse is an example of the general class of Menger-type integrands laid out in \cite{meurer2015fintegral} (the shortened published version is \cite{meurer2018integral}). So, throughout the remainder this section, on a first read one should interpret the phrase $(\mu,p)$-proper integrand, to mean the integrand $\cK_{1}$. The formal definition of a ``$(\mu,p)$-proper integrand" is given in Definition \ref{d:propint}. A symmetric $(\mu,p)$-proper integrand is defined in Definition \ref{d:sympropint}. A review of notation is also included in the preliminaries in section \ref{s:notation}.

One main result of \cite{meurer2015fintegral} is as follows,
\begin{theorem}[\cite{meurer2015fintegral}] \label{t:meurerRect}
Let $E \subset \R^{m}$ be a Borel set and $\mu = \cH^{n} \restr E$. If $\cK$ is a $(\mu,2)$-proper integrand such that
$$
\cM_{\cK^{2}}(\mu) \defeq \int_{\R^{m}} \dots \int_{\R^{m}} \cK^{2}(x_{0}, \dots, x_{n+1}) \dif \mu^{n+2}(x_{0}, \dots, x_{n+1}) < \infty,
$$
then $E$ is countably $n$-rectifiable.
\end{theorem}

\begin{remark}
Meurer also showed that $\cM_{\cK^{2}}(\cH^{n} \restr E) < \infty$ implies $\cH^{n} \restr E$ is locally finite. A similar result does not hold for Borel measures $\mu$ who are absolutely continuous with respect to the Hausdorff measure. For instance, consider $\mu$ on $\R^{2}$ defined by 
$$
\displaystyle \mu = \frac{1}{|x_{1}|} \cH^{1} \restr \{ x_{2} = 0\}
$$ where $x \in \R^{2}$ is written as $x = (x_{1}, x_{2})$. Then $\spt \mu = \{ x_{2} = 0 \}$ implies $\cK_{1}(y_{0}, y_{1},y_{2}) = 0$ whenever $y_{i} \in \spt(\mu)$ for all $i =1,2,3$. Consequently, $\cM_{\cK_{1}}(\mu) = 0$. Nonetheless, $\displaystyle \mu(B(0, \delta)) = \int_{- \delta}^{\delta} \frac{\dif x_{1}}{|x_{1}|} = \infty$ for all $\delta > 0$.
\end{remark}

The next theorem is an interesting intermediate result of \cite{meurer2015fintegral} which is of similar style to the work of Lerman and Whitehouse. As such, one must again have the correct interpretation of the ``space of well-scaled simplices", namely
\begin{equation}
\cO_{k}(x,t) \defeq \{ (x_{0}, \dots, x_{n+1}) \in B(x,kt)^{n+2} \mid |x_{i} - x_{j}| \ge \frac{t}{k} ~ \forall i \neq j \}.
\end{equation} 

Then, adopting the notation
$$
\cM_{\cK^{p};k}(x,t) \defeq \int_{\cO_{k}(x,t)} \cK^{p}(x_{0}, \dots, x_{n+1}) \dif \mu^{n+2}(x_{0}, \dots, x_{n+1})
$$
one can succinctly state:

\begin{theorem}[\cite{meurer2015fintegral}] \label{t:meurerBeta}
Let $\mu$ be an $n$-upper Ahlfors regular Borel measure, with upper-regularity constant $C_{0}$ (see Definition \ref{d:ADreg}). Let $0 < \lambda < 2^{n}$ and $k > 2, k_{0} \ge 1$. Then there exist constants 
\begin{equation*}
\text{$k_{1} = k_{1}(m,n,C_{0}, k,k_{0}, \lambda) > 1$ and $C = C(m,n,\cK,p,C_{0},k,k_{0},\lambda) \ge 1$}
\end{equation*} such that if $\displaystyle{ \mu(B(x,t)) \ge \lambda t^{n}} $, then for every $y \in B(x, k_{0}t)$ we have
$$
\beta_{\mu;p}^{n}(y, kt)^{p} \le C \frac{ \cM_{\cK^{p};k_{1}}(x,t)}{t^{n}} \le C \frac{ \cM_{\cK^{p}, k_{1} + k_{0}}(y,t)}{t^{n}}.
$$
\end{theorem}

A corollary of the previous result, whose relevance to the results of Lerman and Whitehouse is more immediate can be stated as
\begin{corollary}[\cite{meurer2015fintegral}] \label{c:mBeta}
Let $\mu$ be an $n$-upper Ahlfors regular Borel measure on $\R^{m}$ with upper-regularity constant $C_{0}$ . Fix $0 < \lambda < 2^{n}, k > 2, k_{0} \ge 1$ and $\cK^{p}$ any symmetric $(\mu,p)$-proper integrand. Then there exists a constant $C = C(m,n,\cK,p,C_{0}, k, k_{0}, \lambda)$ such that
$$
\int_{\R^{m}} \int_{0}^{\infty} \beta_{\mu;p}^{n}(x,t)^{p} \eye_{\{ \tilde{\delta}_{k}(B(x,t)) \ge \lambda\}} \frac{ \dif  t}{t} \dif \mu(x) \le C \cM_{\cK^{p}}(\mu),
$$
where
$$
\tilde{\delta}_{k}(x,kt) = \sup_{y \in B(x,kt)} \frac{ \mu(B(y,t))}{t^{n}}.
$$
\end{corollary}

\begin{remark}
Although it seemingly goes unmentioned in \cite{meurer2015fintegral}, we note that this corollary implies the following statement:

If $\mu$ is an $n$-Ahlfors regular measure on $\R^{m}$, and $\cM_{\cK^{p}}(\mu|_{B(x,r)}) \lesssim r^{n}$ with suppressed constant independent of $x$, for all $x \in \spt \mu$ and all $0 < r < \diam \spt (\mu)$ for some symmetric $(\mu,p)$-proper integrand, and $p \in  [2, \frac{ 2n}{n-2})$, then $\mu$ is uniformly $n$-rectifiable.

This follows directly from Corollary \ref{c:mBeta} and Theorem \ref{t:betauniformlyrectifiable}. 
\end{remark}

{ In \cite[Equation 10.1]{lerman2011high} Lerman and Whitehouse also introduce the curvature}
$$
\cK^{2}_{2}(x_{0}, \dots, x_{n+1}) = \frac{ h_{\min}(x_{0}, \dots, x_{n+1})^{2}}{\diam(\{x_{0}, \dots, x_{n+1}\})^{n(n+1)+2}},
$$
where 
\begin{equation} \label{e:hmin}
h_{\min}(x_{0}, \dots, x_{n+1}) = \min_{i} \dist(x_{i}, \aff \{x_{0}, \dots, x_{i-1}, x_{i+1}, \dots, x_{n+1}\})
\end{equation}
and $\aff \{y_{0}, \dots, y_{k}\}$ denotes the smallest affine plane containing $\{y_{0}, \dots, y_{k}\}$.

Notably, $\cK_{2}$ satisfies 
$$
\displaystyle \cK^{2}_{1}(x_{0}, \dots, x_{n+1}) \le \cK^{2}_{2}(x_{0}, \dots, x_{n+1})  
$$ 
for all $(x_{0},\dots, x_{n+1}) \in (\R^{m})^{n+2}$ and is also a $(\mu,2)$-proper integrand. We define the pointwise Menger curvature of $\mu$ with respect to a $(\mu,p)$-proper integrand $\cK$ at $x$ and scale $r$ by
\begin{equation} \label{e:pmc}
\curv_{\cK^{p};\mu}^{n}(x,r) = \int_{(B(x,r))^{n+1}} \cK^{p}(x, x_{1}, \dots, x_{n+1}) \dif \mu^{n+1}(x_{1}, \dots, x_{n+1}).
\end{equation}

Then a simplified (and strictly weaker) version of the main result of \cite{kolasinski2016estimating} is:

\begin{lemma} \cite{kolasinski2016estimating} \label{l:kol16}
Let $\mu$ be a Borel measure on $\R^{m}$. Then, there exists $\Gamma = \Gamma(n,m)$ such that
$$
\curv_{\cK^{2}_{2};\mu}^{n}(x,R) \le \Gamma \int_{0}^{2R} \Theta^{n}(\mu,x,r)^{n} \hat{\beta}_{\mu,2}^{n}(x,r)^{2} \frac{ \dif r}{r},
$$

where
$$
\Theta^{n}(\mu,x,r) = \frac{ \mu(B(x,r))}{ r^{n}}
$$
and
\begin{equation} \label{e:cbeta}
\hat{\beta}^{n}_{\mu,p}(x,r)^{p} \defeq \inf_{L \ni x} \frac{1}{r^{n}} \int_{B(x,r)} \left( \frac{ \dist(y, L)}{r} \right)^{p} \dif \mu(y)
\end{equation}
denotes the ``centered" $\beta_{p}$-numbers. Notably, the infimum is taken over all $n$-planes passing through the center of $B(x,r)$. 
\end{lemma}

Another useful result which can be found in a more general setting in \cite{kolasinski2016estimating} is
\begin{lemma} \cite{kolasinski2016estimating} \label{l:samebeta}

Let $\mu$ be a Radon measure on $\R^{m}$ and $x$ be such that $0 < \Theta^{n}_{*}(\mu,x) \le \Theta^{n,*}(\mu,x) < \infty$. Then, for $p \in [1, \infty]$ and $0 < \rho < \infty$, 
\begin{equation} \label{e:k1}
\int_{0}^{\rho} \hat{\beta}^{n}_{\mu;p}(x,r)^{p} \frac{\dif r}{r} < \infty \quad \mu- a.e.~ x \in \R^{m}
\end{equation}
if and only if
\begin{equation} \label{e:k2}
\int_{0}^{\rho} \beta^{n}_{\mu;p}(x,r)^{p} \frac{\dif r}{r} < \infty \quad \mu - a.e. ~ x \in \R^{m}.
\end{equation}

Moreover, if $\mu$ is $n$-Ahlfors regular, and $q \le p$ then there exists $C$ depending on $m,n,p,q$ and the Ahlfors regularity constants such that
\begin{equation} \label{e:k3}
\int_{B(x,r)} \int_{0}^{r} \beta^{n}_{\mu;p}(x,r)^{q} \frac{ \dif r}{r} \le \int_{B(x,r)} \int_{0}^{r} \hat{\beta}^{n}_{\mu;p}(x,r)^{q} \frac{ \dif r}{r} \le C \int_{B(x,C r)} \int_{0}^{C r} \beta_{\mu;p}^{n}(x,r)^{q} \frac{ \dif r}{r}.
\end{equation}
\end{lemma}

A consequence of Lemma \ref{l:kol16} and \cite{tolsa2015characterization} is:
\begin{theorem}[\cite{tolsa2015characterization}, \cite{kolasinski2016estimating}] \label{t:rect2finitecurv}
Let $\mu$ be a Radon measure on $\R^{m}$ with $0 < \Theta^{n}_{*}(\mu,x) \le \Theta^{n,*}(\mu,x) < \infty$ for $\mu$ almost every $x \in \R^{m}$. If $\mu$ is countably $n$-rectifiable, then $\curv^{n}_{\cK_{2}^{2};\mu }(x, 1) < \infty$ for $\mu$ almost every $x \in \R^{m}$.
\end{theorem}

Indeed, the equivalence of (i) and (ii) in Lemma \ref{l:generalityofmain} ensures that Theorem \ref{t:rect2finitecurv} follows from Lemmata \ref{l:kol16}, \ref{l:samebeta}, and Theorem \ref{t:tolchar}.
%
%

\subsection{A new look at rectifiability via Menger curvature}
The first result of this paper is a generalization of Theorem \ref{t:meurerRect} to the case of Radon measures with upper-density bounded above and below.

\begin{theorem} \label{t:radonRect}
If $\mu$ is a Radon measure on $\R^{m}$ with $0 < \Theta^{n,*}(\mu,x) < \infty$ for $\mu$ almost every $x \in \R^{m}$ and $\cM_{\cK^{2}}(\mu) < \infty$ for some $(\mu,2)$-proper integrand $\cK$, then $\mu$ is countably $n$-rectifiable.
\end{theorem}

In light of Theorem \ref{t:radonRect}, with some additional work {we can now answer an open question posed in \cite[Section 6]{lerman2011high}} and characterize $n$-rectifiability with respect to Menger-type curvatures.

\begin{theorem} \label{t:maincharacterization}
Let $\mu$ be a Radon measure on $\R^{m}$ with $0 < \Theta^{n}_{*}(\mu,x) \le \Theta^{n,*}(\mu,x) < \infty$ for $\mu$ almost every $x \in \R^{m}$. Then the following are equivalent: \\
\noindent 1) $\mu$ is countably $n$-rectifiable. \\
\noindent 2) For $\mu$ almost every $x \in \R^{m}$,  $\curv^{n}_{\cK^{2};\mu}(x,1) < \infty$, where $\cK \in \{ \cK_{1}, \cK_{2} \}$. \\
\noindent 3) $\mu$ has $\sigma$-finite integral Menger curvature in the sense that $\mu$ can be written as $\mu = \sum_{j=1}^{\infty} \mu_{j}$ where each $\mu_{j}$ satisfies $\cM_{\cK^{2}}(\mu_{j}) < \infty$ where $\cK \in \{\cK_{1}, \cK_{2}\}$.
\end{theorem}

\begin{remark} \label{r:densityresolution}
Note that Theorem \ref{t:maincharacterization} is to Theorem \ref{t:ENV} as \ref{t:LW} is to \ref{t:betauniformlyrectifiable}, except that unfortunately, the hypothesis in Theorem \ref{t:maincharacterization} are strictly stronger than the hypothesis in Theorem \ref{t:ENV}.  The stronger hypothesis, which may also be an artifact of the proof, does suggest that there may be a better way to define the Menger curvature integrands.
\end{remark}

In particular, with some additional work, combining Theorem \ref{t:ENV}, the equivalence of \eqref{e:k1} with \eqref{e:k2} in Lemma \ref{l:samebeta}, and Theorem \ref{t:maincharacterization} yields another new characterization of rectifiable Radon measures in Theorem \ref{t:i1}. Moreover, combining Theorem \ref{t:betauniformlyrectifiable}, Theorem \ref{t:LW}, and \eqref{e:k3} yields the characterization of uniformly rectifiable measures in Theorem \ref{t:i1}.

\begin{theorem}[\cite{david1991singular}, \cite{lerman2011high}, \cite{lerman2009high}, \cite{meurer2015fintegral}, \cite{kolasinski2016estimating}, Theorem \ref{t:maincharacterization}] \label{t:i1}
If $\mu$ is a Radon measure on $\R^{m}$ with $0 < \Theta^{n}_{*}(\mu,x) \le \Theta^{n,*}(\mu,x) < \infty$ for $\mu$ almost every $x \in \R^{m}$, then the following are equivalent
\begin{enumerate}
\item $\mu$ is countably $n$-rectifiable.
\item $\displaystyle{ \int_{0}^{1} \beta_{\mu;2}^{n}(x,r)^{2} \frac{ \dif r}{r} < \infty}$ for $\mu$ almost every $x \in \R^{m}$.
\item $\displaystyle \int_{0}^{1} \hat{\beta}^{n}_{\mu;2}(x,r)^{2} \frac{ \dif r}{r} < \infty$ for $\mu$ almost every $x \in \R^{m}$.
\item[(4)] $\displaystyle \curv_{\cK^{2}_{2};\mu}^{n}(x,1) < \infty$ for $\mu$ almost every $x \in \R^{m}$.
\item[(5)] $\displaystyle \curv_{\cK^{2}_{1};\mu}^{n}(x,1) < \infty$ for $\mu$ almost every $x \in \R^{m}$.
\end{enumerate}

Moreover, if $\mu$ is an $n$-Ahlfors regular Borel measure on $\R^{m}$ and $p \in [2, \frac{2n}{n-2})$ then the following are equivalent
\begin{enumerate}
\item[(a)] $\mu$ is $n$ uniformly-rectifiable
\item[(b)] $\displaystyle \int_{B(x,R)} \int_{0}^{R} \beta^{n}_{\mu;p}(y,r) \frac{ \dif r}{r} \dif \mu(y) \le C R^{n}$ for all $R>0$.
\item[(c)] $\displaystyle \int_{B(x,R)} \int_{0}^{R} \hat{\beta}^{n}_{\mu;p}(y,r) \frac{ \dif r}{r} \dif \mu(y) \le \tilde{C} R^{n}$ for all $R>0$.
\item[(d)] $\displaystyle \cM_{\cK^{2}_{1}}(\mu \restr B(x,R)) \le C^{\prime} R^{n}$ for all $R>0$.
\item[(e)] $\displaystyle \cM_{\cK^{2}_{2}}(\mu \restr B(x,R)) \le C^{\prime \prime} R^{n}$ for all $R > 0$.
\end{enumerate}
\end{theorem}

The final main result more directly shows a comparability of $\curv_{\cK^{2};\mu}^{n}(x, 1)$ and $\int_{0}^{1}\obeta_{\mu;2}^{n}(x,r) \frac{ \dif r}{r}$, (see \ref{e:pmc} and \ref{e:cbeta} respectively) but in the present state requires stronger hypothesis on the density of $\mu$. This is done by proving a converse to Lemma \ref{l:kol16}

\begin{theorem} \label{t:mainbeta}
If $\mu$ is an $n$-Ahlfors upper-regular Radon measure on $\R^{m}$, $\cK$ is a $(\mu,2)$-proper integrand, and $x$ is such that $\Theta^{n}_{*}(\mu, x) > 0$ then 
\begin{equation} \label{e:boundabove}
\int_{0}^{R} \hat{\beta}^{n}_{\mu;2}(x,r)^{2} \frac{ \dif r}{r} \lesssim \curv^{n}_{\cK^{2}_{1};\mu}(x,R).
\end{equation}
In particular, in conjunction with  Lemma \ref{l:kol16}
\begin{equation} \label{e:i3}
\int_{0}^{R} \obeta_{\mu;2}^{n}(x,r)^{2} \frac{ \dif r}{r} \lesssim \curv^{n}_{\cK^{2}_{1};\mu}(x,R) \le \curv^{n}_{\cK^{2}_{2} ; \mu}(x,R) \lesssim \int_{0}^{C_{1} R} \obeta_{\mu;2}^{n}(x,r)^{2} \frac{ \dif r}{r}.
\end{equation} 
In both \eqref{e:boundabove} and \eqref{e:i3} the suppressed constants\footnote{The dependence on $x$ comes from $\lambda = \lambda_{x,R} > 0$ such that $\lambda r^{n} \le \mu(B(x,r))$ for all $0 < r < R$, see Lemma \ref{l:generalityofmain}.} depend on $x$, $R$, the upper--regularity constant of $\mu$, $m$ and $n$.

\end{theorem}

Since the (suppressed) constant in  \eqref{e:boundabove} depends on $x$, it would be interesting to see if $\mu$ being $n$-Ahlfors upper-regular can be weakened to say $\Theta^{n,*}(\mu,x) < \infty$.

%
%
%
%

\section{Preliminaries}
\subsection{Sets and measures}
When comparing two quantities and the precise constant is unimportant, we adopt the notation that
$$
A \lesssim_{x,r,m} B
$$
means $A \le C B$ for some constant $C$ depending on $x,r,m$. Fewer or more dependencies may be attached to the symbol $\lesssim$. If no dependencies are appended to the symbol, they are explained shortly after the equation appears.
%

\begin{definition}[Ahlfors-regularity of measures] \label{d:ADreg}
A measure $\mu$ on $\R^{m}$ is said to be $n$-Ahlfors regular if there exist constants $0 < c , C < \infty$ such that
\begin{equation} \label{e:p1}
\mu(B(x,r)) \le C r^{n} \qquad \forall x \in \spt(\mu)
\end{equation}
and
\begin{equation} \label{e:p2}
\mu(B(x,r)) \ge c r^{n}  \qquad \forall 0 < r < \diam \{ \spt (\mu) \}, \qquad \forall x \in \spt(\mu).
\end{equation}
A measure $\mu$ is said to be $n$-upper Ahlfors regular if \eqref{e:p1} holds, and $n$-lower Ahlfors regular if \eqref{e:p2} holds. The smallest constant $C$ such that \eqref{e:p1} holds is called the upper regularity constant for $\mu$, and the largest $c$ such that \eqref{e:p2} holds is called the lower regularity constant for $\mu$.
\end{definition}

A measure $\mu$ on $\R^{m}$ is said to be absolutely continuous with respect to a measure $\nu$, denoted $\mu \ll \nu$ if $\nu(E) = 0 \implies \mu(E) = 0$.

\begin{definition}[Countably Rectifiable] \label{d:cr}
In this paper, we follow the convention that a Borel measure $\mu$ on $\R^{m}$ is said to be countably $n$-rectifiable if there exist Lipschitz maps $f_{i} : \R^{n} \to \R^{m}$ such that 
\begin{equation} \label{e:ep3}
\mu \left( \R^{m} \setminus \bigcup_{i=1}^{\infty} f_{i}(\R^{n}) \right) = 0
\end{equation}
and $\mu \ll \cH^{n}$. A Borel set $E$ is countably $n$-rectifiable if $\cH^{n} \restr E$ is countably $n$-rectifiable.
\end{definition}

\begin{definition}[Uniformly rectifiable] \label{d:ur}
A Radon measure $\mu$ on $\R^{m}$ is said to be uniformly $n$-rectifiable if it is $n$-Ahlfors regular and there exist constants $\Lambda > 0$ and $0 < \theta < 1$ such that for all $x \in \spt \mu$ and all $r \ge 0$ there exist a Lipschitz map $f_{x,r} : B^{n}(0,r) \to \R^{m}$ such that 
\begin{equation} \label{e:ep4}
\mu \left( B(x,r) \setminus f_{x,r}(B^{n}(0,r)) \right) \le \theta \mu(B(x,r)).\footnote{Here and after, $B^{n}(0,r)$ denotes an $n$-dimensional ball of radius $r$ centered at $0$. Similarly $B^{m}(x,r)$ is an $m$-dimensional ball centered at $x$ with radius $r$. Typically, the dimension is clear from context and hence neglected. }
\end{equation}
A Borel set $E \subset \R^{m}$ is said to be uniformly $n$-rectifiable if $\cH^{n} \restr E$ is uniformly $n$-rectifiable.
\end{definition}

\begin{definition}[Purely unrectifiable]
A Borel measure $\mu$ on $\R^{m}$ is said to be $n$-purely unrectifiable if every Lipschitz map $f: \R^{n} \to \R^{m}$ has the property that
\begin{equation} \label{e:ep5}
\mu(f(\R^{n})) = 0.
\end{equation}
\end{definition}

A Borel set $E$ is said to be countably $n$-rectifiable (uniformly $n$-rectifiable/$n$-purely unrectifiable respectively) if $\cH^{n} \restr E$ is countably $n$-rectifiable (uniformly $n$-rectifiable/$n$-purely unrectifiable respectively).

\begin{definition}[Density ratios] \label{d:densities}
Given a Borel measure $\mu$ on $\R^{m}$, we define the function
\begin{equation} \label{e:d1}
\Theta^{n}(\mu,x,r) = \frac{ \mu(B(x,r))}{r^{n}}.
\end{equation}
Moreover, the $n$-dimensional upper-density of $\mu$ at $x$, denoted $\Theta^{n,*}(\mu,x)$ is defined as
\begin{equation} \label{e:d2}
\Theta^{n,*}(\mu,x) = \limsup_{r \to 0} \frac{ \mu(B(x,r))}{r^{n}}
\end{equation}
and the $n$-dimensional lower-density of $\mu$ at $x$, denoted $\Theta^{n}_{*}(\mu,x)$ is defined by 
\begin{equation} \label{e:d3}
\Theta^{n}_{*}(\mu,x) = \liminf_{r \to 0} \frac{ \mu(B(x,r))}{r^{n}}
\end{equation}
If $\Theta^{n,*}(\mu,x) = \Theta^{n}_{*}(\mu,x)$ their common value is called the density of $\mu$ at $x$ and is denoted by $\Theta^{n}(\mu,x)$. Notably, $\mu \ll \cH^{n}$ if and only if $\Theta^{n,*}(\mu,x) < \infty$ for $\mu$ a.e. $x \in \R^{m}$. 
\end{definition}

The following lemma is a useful characterization of density properties of measures. 

\begin{lemma} \label{l:generalityofmain}
If $\mu$ is a Borel measure on $\R^{m}$ and $x \in \R^{m}$, then for any $R > 0$, the following are equivalent.
\begin{enumerate}
\item $\Theta^{n}_{*}(\mu,x) > 0$
\item There exists $\lambda > 0$ such that $\mu(B(x,r)) \ge \lambda r^{n}$ for all $0 < r \le R$.
\end{enumerate}

Similarly, the following are equivalent:
\begin{enumerate}
\item[(i)] $\Theta^{n,*}(\mu,x) < \infty$ 
\item[(ii)] There exists a $\Lambda > 0$ such that $\mu(B(x,r)) \le \lambda r^{n}$ for all $0 < r \le R$.
\end{enumerate}
\end{lemma}

\begin{proof} 
We only discuss the proof that (1) and (2) are equivalent. The proof that (i) and (ii) are equivalent follows the same structure. 

First, note that (2) implies $\Theta^{n}_{*}(\mu,x) \ge \lambda$. So, we assume (1) and show (2). Since $\Theta^{n}_{*}(\mu,x) > 0$ it follows that there exists $\delta = \delta(x)$ such that for all $r \le \delta, \mu(B(x,r)) \ge \frac{\Theta^{n}_{*}(\mu,x)}{2} r^{n}$. In particular,
$$
\mu(B(x, \delta)) \ge \frac{ \Theta^{n}_{*}(\mu,x)}{2} \delta^{n},
$$
so for $\delta \le r \le R$ it follows
$$
\mu(B(x,r)) \ge \mu(B(x,\delta)) \ge \frac{ \Theta^{n}_{*}(\mu,x)}{2} \delta^{n} \ge \frac{ \Theta^{n}_{*}(\mu,x)}{2} \frac{\delta^{n}}{R^{n}} r^{n},
$$
so $\lambda = \frac{ \Theta^{n}_{*}(\mu,x)}{2} \left(\frac{\delta}{R}\right)^{n} \le \Theta^{n}_{*}(\mu,x)/2$ suffices.
\end{proof}

Finally, given a measure $\mu$ on $\R^{m}$ we let $\mu^{k}$ denote the measure on $(\R^{m})^{k}$ defined as the $k$-fold product of $\mu$ with itself. Similarly, given a set $E \subset \R^{m}$ we let $E^{k}$ denote the $k$-fold product of $E$ as a set in $(\R^{m})^{k}$.

\subsection{Menger-type curvature, a formal review} \label{s:MCreview}
\subsubsection*{Simplices and Notation} \label{s:notation}
Given points $\{x_{0}, \dots, x_{n}\} \subset \R^{m}$ then $\Delta(x_{0}, \dots, x_{n})$ will denote the convex hull of $\{x_{0}, \dots, x_{n}\}$. In particular, if $\{x_{0}, \dots, x_{n}\}$ are not contained in any $(n-1)$-dimensional plane, then $\Delta(x_{0}, \dots, x_{n})$ is an $n$-dimensional simplex. Moreover, $\aff\{x_{0}, \dots, x_{n}\}$ denotes the smallest affine subspace containing $\{x_{0}, \dots, x_{n}\}$. That is $\aff \{x_{0}, \dots, x_{n} \} = x_{0} + \Span \{x_{1}- x_{0}, \dots, x_{n} - x_{0} \}$.

If $\Delta$ is an $n$-simplex, it is additionally called an $(n, \rho)$-simplex if 
$$
h_{\min}(x_{0}, \dots, x_{n}) = \min_{i} \dist(x_{i}, \aff\{x_{0}, \dots, x_{i-1}, x_{i+1}, \dots, x_{n} \}) \ge \rho.
$$

The next lemma can be found in \cite[Lemma 3.7]{meurer2018integral} and quantifies some geometric properties of simplices, which is of particular interest when showing a certain integrand is $(\mu,p)$-proper. 

\begin{lemma} \label{l:sg} Let $C \ge 1, t > 0, x \in \R^{m}, w \in B(x,Ct)$ and $S = \Delta(x_{0}, \dots, x_{n}) \subset B(x,Ct)$ be some $(n, \frac{t}{C})$-simplex. Define $S_{w} = \Delta(x_{0}, \dots, x_{n}, w)$ and choose distinct $i,j \in \{0, \dots, n \}$. Then
\begin{enumerate}
\item $\frac{t}{C} \le |x_{i} - x_{j}| \le \diam(S_{w}) \le 2Ct$,
\item $|x_{i} - w| \le 2 Ct$,
\item $\frac{t^{n}}{C^{n} n!} \le \cH^{n}(S) \le \frac{(2C)^{n}}{n!} t^{n}$,
\item $\dist(w, \aff(x_{0}, \dots, x_{n})) = n \frac{\cH^{n+1}(S_{w})}{\cH^{n}(S)}$
\end{enumerate}
\end{lemma}

The next lemma is an immediate consequence of repeated applications of  \cite[Lemma 2.17]{meurer2018integral}.

\begin{lemma} \label{l:boundedheightchange}
Let $0 < k \le n$. If $T_{x} = \Delta(x_{0}, \dots, x_{n})$ is an $(n, \rho)$-simplex, and $\{y_{0}, \dots, y_{k}\}$ are such that $|x_{i} - y_{i}| < \delta$ for some $(k+1) \delta < \rho$, then $\Delta(y_{0}, \dots, y_{k}, x_{k+1}, \dots, x_{n})$ is an $(n, \rho - (k+1)\delta)$-simplex.
\end{lemma}


\subsubsection*{Meurer's proper integrands}

The ability to use a general class of integrands $\cK : (\R^{m})^{n+2} \to [0, \infty)$ as a tool to study countably $n$-rectifiable sets and measures in $\R^{m}$ was demonstrated by Meurer in \cite{meurer2018integral}. Below is the definition of the general class of integrands as laid out by Meurer.

\begin{definition}[$(\mu,p)$-proper integrand] \label{d:propint}
Let $n,m \in \N$ with $1 \le n < m$. Let $\cK : (\R^{m})^{n+2} \to [0, \infty)$ and $p > 1$. One says that $\cK$ is a $(\mu,p)$-proper integrand if it fulfills the following four conditions:
\begin{enumerate}
\item $\cK$ is $\mu^{n+2}$-measurable, where $\mu^{n+2}$ denotes the $(n+2)$-product measure of $\mu$.
\item There exists some constants $c = c(n,\cK,p) \ge 1$ and $\ell = \ell(n, \cK,p) \ge 1$ so that, for all $t > 0$, $C \ge 1, x \in \R^{m}$ and all $(n, \frac{t}{C})$-simplices $\Delta(x_{0}, \dots, x_{n}) \subset B(x,Ct)$, it follows
\begin{equation} \label{e:p303}
\left( \frac{ d(w, \aff (x_{0}, \dots, x_{n}))}{t} \right)^{p} \le c C^{\ell} t^{n(n+1)} \cK^{p}(x_{0}, \dots, x_{n}, w)
\end{equation}
for all $w \in B(x,Ct)$.
\item For all $\lambda > 0$, 
\begin{equation} \label{e:p3}
\lambda^{n(n+1)} \cK^{p}(\lambda x_{0}, \dots, \lambda x_{n+1}) = \cK^{p}(x_{0}, \dots, x_{n+1})
\end{equation}
\item $\cK$ is translation invariant in the sense that for every $b \in \R^{m}$, 
\begin{equation} \label{e:p4}
\cK(x_{0} + b, \dots, x_{n+1} + b) = \cK(x_{0}, \dots, x_{n+1})
\end{equation}
\end{enumerate}
\end{definition}

\begin{remark}
The preceding definition is rather long and written so that expressions show-up in the same form that they do in the proof of \cite[Theorem 5.6]{meurer2015fintegral}, a theorem which roughly provides a bound on $\beta_{p}$-numbers by Menger curvature.  As written above, one may notice that part (2) looks vaguely like one is bounding $\beta_{p}$-numbers. However, the relationship becomes more obvious after applying part 3 to re-write part 2 of the definition of a $(\mu,p)$-proper integrand in the following way:

There exists some constant $c = c(n,\cK,p) \ge 1$ and $\ell = \ell(n, \cK,p) \ge 1$ so that, for all $t > 0$, $C \ge 1, x \in \R^{m}$ and all $(n, \frac{t}{C})$-simplices $\Delta(x_{0}, \dots, x_{n+1}) \subset B(x,Ct)$, it follows
$$
\left( \frac{ d(w, \aff (x_{0}, \dots, x_{n}))}{t} \right)^{p} \le c C^{\ell} \cK^{p} \left( \frac{x_{0}}{t}, \dots, \frac{x_{n+1}}{t}, \frac{w}{t} \right)
$$
for all $w \in B(x,Ct)$.
In particular, ignoring all details and technicalities it looks like integrating the left-hand side over $w$ yields the $L^{p}$-distance to a specific plane at scale $t$ is bounded by the Menger curvature integrand ``at scale $t$" when integrated over just one input, while the other inputs span the given affine plane. 
\end{remark}

Given a Borel measure $\mu$ and a $(\mu,p)$-proper integrand $\cK$, the integral Menger curvature of $\mu$ with respect to $\cK$ (or simply integral Menger curvature) is
$$
\cM_{\cK^{p}}(\mu) = \int_{(\R^{m})^{n+2}} \cK^{p}(x_{0}, \dots, x_{n+1}) \dif \mu^{n+2}(x_{0}, \dots, x_{n+1}).
$$
The pointwise Menger curvature of $x$ in $\mu$ with respect to $\cK$ at scale $r$ is
$$
\curv_{\cK^{p};\mu}^{n}(x,r) = \int_{B(x,r)^{n+1}} \cK^{p}(x,x_{1}, \dots, x_{n+1}) \dif \mu^{n+1}(x_{1}, \dots, x_{n+1}).
$$

We next show why one of the two integrands emphasized throughout this paper does indeed satisfy the definition of a $(\mu,2)$-proper integrand. As noted in \cite[Lemma 3.9]{meurer2015fintegral}, the computations here are analogous to \cite[Lemmata 3.7 and 3.8]{meurer2015fintegral}. Nevertheless, we include them for the reader's convenience.

To precisely express the two integrands we define 
$$
X = \{ (x_{0}, \dots, x_{n+1}) : \cH^{n+1}(\Delta(x_{0}, \dots, x_{n+1}))\} > 0.
$$

\begin{example} \label{x:propint} The following integrands are $(\mu,2)$-proper although they first appear in the works of Lerman and Whitehouse (see \cite{lerman2009high}, \cite{lerman2011high}, \cite{lerman2012least}).
$$
\cK_{1}(x_{0}, \dots, x_{n+1}) \defeq \eye_{X}(x_{0}, \dots, x_{n+1}) \frac{ \cH^{n+1}(\Delta(x_{0}, \dots, x_{n+1}))}{\left(\diam \{x_{0}, \dots, x_{n+1}\} \right)^{\frac{(n+1)(n+2)}{2}}}
$$
and
\begin{equation*} \label{e:kolint}
\cK_{2}(x_{0}, \dots, x_{n+1}) = \eye_{X}(x_{0}, \dots, x_{n+1}) \frac{ h_{\min}(x_{0}, \dots, x_{n+1})}{\diam(\{x_{0}, \dots, x_{n+1}\})^{\frac{n(n+1)+2}{2}}}
\end{equation*}
where $h_{\min}(x_{0}, \dots, x_{n+1})) = \min_{i} \dist \{ x_{i}, \aff \{x_{0}, \dots, x_{i-1}, x_{i+1}, \dots, x_{n+1}\}$.

Since the expressions for the Menger-type curvatures look like they tend to zero as $(x_{0}, \dots, x_{n+1})$ tend to $X^{c}$, the indicator function is typically neglected. Below is an outline of why (by considering $\cK_{1}$) these examples are proper integrands. The proof of $\cK_{2}$ is more straightforward.

Measurability follows due to the fact that $X$ and $(\R^{m})^{n+2} \setminus X$ are open and closed respectively in $(\R^{m})^{n+2}$. So $\mu^{n+2}$-measurability follows since $\mu$ is Borel and $\cK_{1}$ is continuous on $X$ and $(\R^{m})^{n+2} \setminus X$. 

For the second condition in the definition of $(\mu,2)$-proper, consider $t > 0, C \ge 1, x \in \R^{m}, \Delta = \Delta(x_{0}, \dots, x_{n}) \subset B(x,Ct)$ is an $(n, \frac{t}{C})$-simplex, fix $w \in B(x,Ct)$ and let $\Delta_{w} = \Delta(x_{0}, \dots, x_{n}, w)$. Then,
\begin{align*}
\tag{\ref{l:sg} part (5)} \left( \frac{ d(w, \aff(x_{0}, \dots, x_{n}))}{t} \right)^{2} & = \left( n \frac{ \cH^{n+1}(\Delta_{w})}{\cH^{n}(\Delta)} \right)^{2} \\
\tag{\ref{l:sg} part (3)} & \le (n \cdot n! \cdot C^{n})^{2} \left( \frac{ \cH^{n+1}(\Delta_{w})}{t \cdot t^{n}} \right)^{2} \\
& = (n \cdot n! \cdot C^{n})^{2} t^{n(n+1)} \left( \frac{ \cH^{n+1}(\Delta_{w})}{t^{(n+1) + \frac{n(n+1)}{2}}} \right)^{2} \\
& = (n \cdot n! \cdot C^{n})^{2} t^{n(n+1)} \left( \frac{ \cH^{n+1}(\Delta_{w})}{t^{\frac{(n+1)(n+2)}{2}}} \right)^{2} \\
\tag{\ref{l:sg} parts (1,2)} & \le (n \cdot n! \cdot C^{n})^{2} C^{\frac{(n+1)(n+2)}{2}} t^{n(n+1)} \cK^{2}_{1}(\Delta)
\end{align*}
hence, the second property holds with $\ell = \frac{(n+2)(n+1)}{2} + 2n$ and $c = (n \cdot N!)^{2}$.

For homogeneity, note that if $\lambda > 0$, then $(x_{0}, \dots, x_{n+1}) \in X \iff (\lambda x_{0}, \dots, \lambda x_{n+1}) \in X$. Moreover, for $(x_{0}, \dots, x_{n+1}) \in X$ it follows that 
$$
\cH^{n+1}(\lambda x_{0}, \dots, \lambda x_{n+1}) = \lambda^{n+1} \cH^{n+1}(x_{0}, \dots, x_{n+1}).
$$ 
Consequently
$$
\cK^{2}_{1}(\lambda x_{0} \dots, \lambda x_{n+1}) = \frac{ \lambda^{2(n+1)}}{\lambda^{(n+2)(n+1)}} \cK^{2}_{1}(x_{0}, \dots, x_{n+1}) = \lambda^{-n(n+1)} \cK^{2}_{1}(x_{0}, \dots, x_{n+1}).
$$ 
Translation invariance follows from the geometric nature of the definition.
\end{example}

\begin{definition}[Symmetric $(\mu,p)$-proper integrand] \label{d:sympropint}
A $(\mu,p)$-proper integrand is said to be symmetric if for all permutations $\sigma \in S_{n+2}$
$$
\cK^{p}(x_{0}, \dots, x_{n+1}) = \cK^{p}(x_{\sigma(0)}, \dots, x_{\sigma(n+1)})
$$
\end{definition}

The next lemma, due to \cite[Lemma 5.1]{meurer2018integral}, demonstrates that the restriction to symmetric proper integrands is a non-issue.

\begin{lemma} \label{l:p1}
Let $\mu$ be a Radon measure on $\R^{m}$ and fix $\cK^{p}$ some $(\mu,p)$ proper integrand. Then, there exists $\tilde{\cK}^{p}$ a symmetric $(\mu,p)$-proper integrand which satisfies $\cM_{\cK^{p}}(\mu \cap E ) = \cM_{\tilde{\cK}^{p}}(\mu \cap E)$ for all Borel sets $E$.
\end{lemma}

The proof is to use Fubini's theorem to check that the integrand
$$
\tilde{\cK}^{p}(x_{0}, \dots, x_{n+1}) = \frac{1}{\#|S_{n+2}|} \sum_{\sigma \in S_{n+2}} \cK^{p}(x_{\sigma(0)}, \dots, x_{\sigma(n+1)})
$$
satisfies $\cM_{\cK^{p}}(\mu \cap E) = \cM_{\tilde{\cK}^{p}}(\mu \cap E)$ for all Borel $E$. Moreover, it clearly satisfies conditions (1), (3), and (4) in Definition \ref{d:propint}. So, it only remains to check that condition (2) holds. But, $\cK^{p} \le \# |S_{n+2}| \tilde{\cK}^{p}$ validates condition (2) of Definition \eqref{d:propint}.


\section{Proofs of main results}
One main tool of Meurer's work (see \cite[Theorem 4.1]{meurer2015fintegral}), a non-trivial generalization of an analogous result from \cite{leger1999menger}, is the following:
\begin{theorem} \label{t:maingraph} 
Let $\cK : (\R^{m})^{n+2} \to [0, \infty)$ be a $(\mu,2)$-proper integrand. Suppose $\mu$ is a Borel measure on $\R^{m}$. Then, there exists some small $\eta = \eta ( \cK, n,m, C_{0}) > 0$, so that if $\mu$ satisfies
\begin{enumerate}
\item[(A)] $\mu(B(0,2)) \ge 1$ and $\mu(\R^{m} \setminus B(0,2)) = 0$
\item[(B)] $\mu(B) \le C_{0} ( \diam B)^{n}$ for every ball $B$.
\item[(C)] $\cM_{\cK^{2}}(\mu) \le \eta$
\end{enumerate}
then there exists some Lipschitz function $f : \R^{n} \to \R^{m-n}$ with Lipschitz constant bounded above by some $\Lambda = \Lambda(\cK,n,m,C_{0})$ such that a rotation of the graph of $f$, named $\Gamma$ satisfies
\begin{equation} \label{e:MGT}
\mu( \R^{m} \setminus \Gamma) < \frac{1}{100} \mu(\R^{m}).
\end{equation}
Moreover, for given $\cK$ and $C_{0}$ the Lipschitz constant of $f$ goes to zero as $\cM_{\cK^{2}}(\mu)$ approaches zero.
\end{theorem}

\subsection{Scaling Menger Curvature}
Theorem \ref{t:maingraph} is one of the main tools for proving Theorem \ref{t:radonRect}. It will also be useful to know how integral Menger curvature scales, and how this impacts Theorem \ref{t:maingraph}.


\begin{proposition} \label{p:MCTM}
Let $\mu$ be a Radon measure and $\cK$ a $(\mu, p)$-proper integrand.  Let $\nu$ be the Radon measure defined by $\nu(A) = \lambda \mu( aA + x)$ for some $a, \lambda > 0$ and $x \in \R^{m}$ then 
\begin{equation} \label{e:s1}
\cM_{\cK^{2}}(\nu) = \lambda^{n+2} a^{-n(n+1)} \cM_{\cK^{2}}(\mu).
\end{equation}

In particular, if $\mu_{x,r}$ is defined so that $\mu_{x,r}(E) = \frac{\mu(rE + x)}{r^{n}}$ for all $E \subset \R^{m}$, then
\begin{equation} \label{e:s2}
 \cM_{\cK^{p}}( \mu_{x,r} \restr B(0,1)) = \frac{\cM_{\cK^{p}}(\mu \restr B(x,r)) }{r^{n}}
\end{equation}
\end{proposition}

\begin{proof} Fix a Borel measure $\mu$ on $\R^{m}$ a point, $x\in \R^{m}$ and $a, \lambda > 0$. Define $f : \R^{m} \to \R^{m}$ by $f(y) = \frac{y-x}{a}$, then $\nu = \lambda f_{\#} \mu$ (See \cite[1.17-1.19]{mattila1999geometry} for definition and use of image measures). Consequently,
\begin{align*}
\cM_{\cK^{p}} ( \nu ) & = \int_{(\R^{m})^{n+2}} \cK^{p}(y_{0}, \dots, y_{n+1}) \dif (\lambda f_{\#} \mu(y_{0})) \cdots \dif ( \lambda f_{\#} \mu(y_{n+1})) \\
& = \lambda^{n+2} \int_{(\R^{m})^{n+2}} \cK^{p} \left( ax_{0} + x, ax_{1} + x, \dots, a x_{n+1} + x \right) \dif \mu(x_{0}) \cdots \dif \mu(x_{n+1}) \\
& = \lambda^{n+2} a^{-n(n+1)} \int_{(\R^{m})^{n+2}} \cK^{p}(x_{0}, \dots, x_{n+1}) \dif \mu(x_{0}) \cdots \dif \mu(x_{n+1}) \\
\end{align*}
where the final line follows by first applying translation invariance and then the homogeneity of $\cK^{p}$ (see conditions \eqref{e:p3} and \eqref{e:p4} in the definition of a $(\mu,p)$-proper integrand). This proves \eqref{e:s1}.

We see \eqref{e:s2} follows from choosing $\lambda = r^{-n}$ and $a = r^{-1}$ combined with the observation that $f^{-1}(B(0,1)) = B(x,r)$ when $f(y) = \frac{y-x}{r}$.
\end{proof}

We note that \eqref{e:s2} is of interest due to the following modification of Theorem \ref{t:maingraph}

\begin{theorem} \label{t:s1}
Let $\mu$ be an $n$-Ahlfors upper-regular Radon measure with upper-regularity constant $C$. Let $\cK$ be a $(\mu,2)$-proper integrand. Then, there exists a function $\eta_{1}= \eta_{1}( \cK, n,m, \Theta^{n}(\mu,x,r),C) > 0$ such that
\begin{equation} \label{e:s3}
\frac{ \cM_{\cK^{2}} \left( \mu \restr B(x,r) \right)}{r^{n}} \le \eta_{1}
\end{equation}
implies there exists some Lipschitz graph $\Gamma$  with Lipschitz constant bounded above by some $\Lambda = \Lambda(\eta_{1})$ such that
\begin{equation} \label{e:s4}
\mu \left( B(x,r) \setminus \Gamma \right) < \frac{1}{100} \mu(B(x,r)).
\end{equation}

Moreover, given $\cK, \Theta^{n}(\mu,x,r), C$ the Lipschitz constant of $\Gamma$ tends to zero as 
$\displaystyle{\frac{ \cM_{\cK^{2}}(\mu \restr B(x,r))}{r^{n}}}$ approaches zero.
\end{theorem}

An immediate corollary to Theorem \ref{t:s1} is

\begin{corollary} \label{c:s1}
Let $\mu$ be an $n$-Ahlfors regular Radon measure with lower-regularity constant $c$, and upper-regularity constant $C$. If $\cK$ is a $(\mu,2)$-proper integrand and \eqref{e:s3} is satisfied with $\eta_{1} = \eta_{1} (\cK, n,m, c,C)$ for all $x \in \spt(\mu)$ and all $0 < r < \diam( \spt(\mu))$ then $\mu$ is uniformly $n$-rectifiable.
\end{corollary}

\begin{proof} (of Theorem \ref{t:s1}). We claim that
$\displaystyle{\nu = \frac{\mu_{x,r} \restr B(0,1)}{\Theta^{n}(\mu,x,r) }}$ satisfies 
\begin{enumerate}
\item $\displaystyle \nu(B(0,1)) \ge 1$ and $\displaystyle \nu(\R^{m} \setminus B(0,1)) = 0$.
\item $\displaystyle \nu(B(y,t)) \le \frac{C}{c}t^{n}$
\item 
$$
\cM_{\cK^{2}}(\nu) = \frac{\cM_{\cK^{2}} \left( \mu \restr B(x,r) \right)}{r^{n} c^{n+2}}.
$$
\end{enumerate}

Indeed, (1) follows since $\nu(B((0,1))) = \frac{1}{\Theta^{n}(\mu,x,r)} \frac{ \mu(B(x,r))}{r^{n}} = \frac{ \Theta^{n}(\mu,x,r)}{\Theta^{n}(\mu,x,r)} = 1$ and $\nu$ is the restriction of a measure to $B(0,1)$.

To see (2) follows, consider $y \in B(0,1)$ and $t >0$. Then, 
\begin{align*}
\nu(B(y,t)) &= \frac{1}{\Theta^{n}(\mu,x,r)} \frac{\mu \left( B(x+ry,rt) \cap B(x,r) \right)}{r^{n}} \\
& \le \frac{1}{\Theta^{n}(\mu,x,r)} \frac{ \mu(B(x+ry, rt))}{r^{n}} \le \frac{1}{\Theta^{n}(\mu,x,r)} \frac{ C (rt)^{n}}{r^{n}} \\
& \le \frac{C}{\Theta^{n}(\mu,x,r)} t^{n}.
\end{align*}

Finally, (3) follows since \eqref{e:s2} in Proposition \ref{p:MCTM} guarantees 
$$
\cM_{\cK^{2}} \left( \mu_{x,r} \restr B(0,1) \right) = \frac{ \cM_{\cK^{2}}(\mu \restr B(x,r))}{r^{n}}
$$
then applying \eqref{e:s1} with $\lambda = \Theta^{n}(\mu,x,r)^{-1}$, $a =1$ yields
\begin{align*}
\cM_{\cK^{2}} (\nu) &= \cM_{\cK^{2}}( \Theta^{n}(\mu,x,r)^{-1} \mu_{x,r} \restr B(0,1)) = \Theta^{n}(\mu,x,r)^{-(n+2)} \cM_{\cK^{2}} \left( \mu_{x,r} \restr B(0,1) \right) \\
& = \Theta^{n}(\mu,x,r)^{-(n+2)} \frac{ \cM_{\cK^{2}}(\mu \restr B(x,r))}{r^{n}}
\end{align*}

Consequently, if $\eta_{1} \left( \cK, m,n, \Theta^{n}(\mu,x,r),C \right)$ is chosen so that it is at most as large as $\Theta^{n}(\mu,x,r)^{n+2} \eta( \cK, m,n, \frac{C}{c})$ where $\eta(\cK,m,n\frac{C}{c})$ is as in Theorem \ref{t:maingraph}, then $\nu$ satisfies the hypothesis of Theorem \ref{t:maingraph}. Hence, there exists a Lipschitz graph $\tilde{\Gamma}$ with $\nu( \R^{m} \setminus \tilde{\Gamma}) < \frac{1}{100} \nu(\R^{m})$. But $\frac{\nu(\R^{m} \setminus \tilde{\Gamma})}{\nu(\R^{m})} = \frac{\mu(B(x,r) \setminus \Gamma)}{\mu(B(x,r))}$ where $\Gamma = r\Tilde{\Gamma} + x$. In particular, $\mu(B(x,r) \setminus \Gamma) < \frac{1}{100} \mu(B(x,r))$ as desired.
\end{proof}

\begin{remark}
Note that in the case $\cK = \cK_{1}$ (along with several other specific integrands studied in \cite{lerman2009high}, \cite{lerman2011high}) Corollary \ref{c:s1} is already known from the work of Lerman and Whitehouse. In fact, the hypothesis in Theorem \ref{t:s1} is stronger than theirs, because Theorem \ref{t:s1} requires not only a Carleson-type bound on the local integral Menger curvature, but that the bound be by a small constant. 

To create an effective theory of quantitative, albeit non-uniform rectifiability, it would be interesting to address this seemingly unnecessary ``smallness" condition being imposed by $\eta_{1}$ in \eqref{e:s3}. It would also likely be useful to allow $\Gamma$ to only satisfy $\mu(B(x,r) \setminus \Gamma) \le (1 - \epsilon) \mu(B(x,r))$ for some $\epsilon = \epsilon(\eta,m,n,\cK,C, \Theta^{n}(\mu,x,r))$.
\end{remark}


\subsection{Rectifiability from integral Menger curvature}

The goal of this section is to prove Theorem \ref{t:radonRect} included below for completeness. 

\begin{theorem} \label{t:mainrect}
If $\mu$ is a Radon measure on $\R^{m}$ with $0 < \Theta^{n,*}(\mu,x) < \infty$ for $\mu$ almost every $x \in \R^{m}$ and $\cM_{\cK^{2}}(\mu) < \infty$ for some $(\mu,2)$-proper integrand $\cK$, then $\mu$ is countably $n$-rectifiable.
\end{theorem}

The proof could be separated into two parts. The first part is a sequence of arguments to show that Radon measures which are mutually absolutely continuous with respect to the Hausdorff measure behave sufficiently similar the Hausdorff measure restricted to sets. The second part follows an argument from Sections 1 and 2 of \cite{leger1999menger} and is also similar to the case for sets as presented in \cite{meurer2018integral}. Several preparatory lemmas are required.

\begin{lemma} \label{l:step0}
Let $\mu$ be a Radon measure on $\R^{m}$ with $0 < \Theta^{n,*}(\mu,x) < \infty$ for $\mu$ almost every $x \in \R^{m}$. Then there exists some Borel set $E$ such that $\nu = \mu \restr E$ satisfies  $\nu(\R^{m}) \ge \frac{1}{2} \mu(\R^{m})$ and 
$$
0 < c \le \Theta^{n,*}(\nu, x) \le C < \infty \text{ for $\nu$ almost every $x \in \R^{m}$.}
$$
\end{lemma}

\begin{proof}
Without loss of generality, suppose $\mu(\R^{m}) < \infty$. Otherwise, apply the finite case to a family of disjoint annulli that exhaust $\R^{m}$. Moreover, suppose $\mu(\R^{m}) > 0$ to avoid trivialities.

Define $E_{j}  = \{ x \in \R^{m} : 2^{-j} \le \Theta^{n,*}(\mu,x) \le 2^{j} \}$. Then  $E_{j}^{c} \supset E_{j+1}^{c}$ for all $j$ and moreover 
$$
\bigcap_{j=1}^{\infty} E_{j}^{c} = \{ x \in \R^{m} : 0 = \Theta^{n,*}(\mu,x) \text{ or } \Theta^{n,*}(\mu,x) = + \infty \}.
$$
Since $\mu(E_{1}^{c}) \le \mu(\R^{m}) < \infty$ it follows that
$$
\lim_{j \to \infty} \mu(E_{j}^{c}) = \mu \left( \bigcap_{j=1}^{\infty} E_{j}^{c} \right) = 0
$$
so there exists $k$ with $\mu(E_{k}) \ge \frac{1}{2} \mu(\R^{m})$. Fix such $k$ and let $c = 2^{-k}$ and $C = 2^{k}$. Then, define $\nu = \mu \restr E_{k}$. Since $E_{k}$ is $\mu$-measurable, it follows $\nu$ is Radon and $\nu(\R^{m}) \ge \frac{1}{2} \mu(\R^{m})$. On the other hand the Lebesgue-Besicovitch differentiation theorem ensures $\Theta^{n,*}(\nu,x) = \Theta^{n,*}(\mu,x)$ for $\nu$ a.e. $x \in \R^{m}$ since $\frac{ \nu(B(x,r))}{\mu(B(x,r))} \xrightarrow{r \downarrow 0} 1$ for $\nu$ a.e. $x \in \R^{m}$. Consequently $2^{-k} \le  \Theta^{n,*}(\nu, x) \le 2^{k}$ for $\nu$ a.e. $x \in \R^{m}$ as desired.
\end{proof}

\begin{lemma} \label{l:53}
If $\mu$ is a Radon measure with $0 < c \le \Theta^{n,*}(\mu,x) \le C < \infty$ for $\mu$ a.e. $x \in \R^{m}$ and $\spt(\mu)$ is bounded and $\cM_{\cK^{p}}(\mu) < \infty$. Then, for all $\zeta_{0} > 0$ there exists a compact set $E^{*} \subset \spt(\mu)$ with
\begin{enumerate}
\item[(i)] $\mu(E^{*}) \ge \frac{c}{2^{n+2}} \left( \diam E^{*} \right)^{n}$
\item[(ii)] For all $x \in E^{*}$ and all $t > 0$, $\mu \left( E^{*} \cap B(x,t) \right) \le 2 C t^{n}$.
\item[(iii)] $\cM_{\cK^{p}}(\mu \restr E^{*}) \le \zeta_{0} \left( \diam E^{*} \right)^{n}$
\end{enumerate}
\end{lemma}

\begin{proof}
Since $\mu$ is Radon, and $\spt(\mu)$ is bounded, it follows that $\mu(\R^{m}) < \infty$. Define
\begin{equation} \label{e:EL} 
E_{\ell} = \{ x \in \R^{m} : \forall t \in (0, 2^{-\ell}), ~ \mu(B(x,t)) \le 2 C t^{n} \}.
\end{equation} 

Evidently, $E_{\ell} \subset E_{\ell+1}$. Moreover, the assumption $c \le \Theta^{n,*}(\mu,x) \le C$ for $\mu$ almost every $x \in \R^{m}$ ensures 
\begin{equation*} 
\mu(E_{\ell}) \xrightarrow{\ell \to \infty} \mu(\R^{m}).
\end{equation*}

Hence, by the finiteness of $\mu(\R^{m})$ there exists some $\ell$ such that $\mu(E_{\ell}) \ge \frac{1}{2} \mu(\R^{m})$. Define $\nu = \mu \restr E_{\ell}$. Then, notably
\begin{equation} \label{e:smallupperreg}
\nu(B(x,r)) \le 2C r^{n} \qquad \forall x \in E_{\ell} \qquad \forall 0 < r \le 2^{- \ell}
\end{equation}
and by Lebesgue-Besicovitch differentiation theorem \cite[Theorem 1.7.1]{gariepy1992measure}, it also follows 
\begin{equation} \label{e:smalllowerreg}
c \le \Theta^{n,*}(\nu,x) \qquad \nu ~ a.e. ~ x \in E_{\ell}
\end{equation}
since $\displaystyle{\frac{ \mu(B(x,r) \cap E_{\ell})}{\mu(B(x,r))} \xrightarrow{r \to 0} 1}$ for $\mu$ a.e. $x \in E_{\ell}$ and by assumption $\displaystyle \limsup_{r \to 0} \frac{ \mu(B(x,r) }{r^{n}} \ge c$ for $\mu$ almost every $x \in \R^{m}$. 

Define 
\begin{equation} \label{e:Atau}
A(\tau) = \left\{ (x_{0}, \dots, x_{n+1}) \in (E_{\ell})^{n+2} : |x_{0} - x_{i}| < \tau ~ \forall i \in \{1, \dots, n+1\} \right\}.
\end{equation}

Claim 1: $\nu^{n+2}(A(\tau)) \xrightarrow{\tau \to 0} 0$.

Indeed, note that $A(\tau) = \bigcup_{x \in E_{\ell}} \{x\} \times (B(x, \tau)\cap E_{\ell})^{n+1}$. In particular \eqref{e:smallupperreg} ensures that  for all $\tau \le 2^{-\ell}$,
$$
\nu^{n+2}(A(\tau)) = \int_{E_{\ell}} \nu^{n+1}(B(x,\tau) \cap E_{\ell}) \dif \nu(x) \le \int_{E_{\ell}} 2C \tau^{n(n+1)} \dif \nu(x) = 2C \tau^{n(n+1)} \nu(E_{\ell}).
$$
Since $\nu$ is finite, the claim follows.

Claim 2: $I(\tau)$ defined in \eqref{e:Itau} satisfies $I(\tau) \xrightarrow{\tau \to 0} 0$.

\begin{equation} \label{e:Itau}
I(\tau) = \int_{A(\tau)} \cK^{p}(x_{0}, \dots, x_{n+1}) \dif \nu^{n+2}(x_{0}, \dots, x_{n+1}) 
\end{equation}

Indeed, we first note, $\cM_{\cK^{p}}(\nu) \le \cM_{\cK^{p}}(\mu) < \infty$ implies $\displaystyle{ \cK^{p} \in L^{1}((\R^{m})^{n+2}, \nu^{n+2})}$. Then write
$$
I(\tau) = \int_{(\R^{m})^{n+2}} \eye_{A(\tau)} \cK^{p} \dif \nu^{n+2} \le \int_{(\R^{m})^{n+2}} \cK^{p} \dif \nu^{n+2}.
$$ 
Consequently, for any sequence of $\tau_{k}$ converging to zero, Claim 1 ensures that the corresponding sequence of functions $\{ \eye_{A(\tau_{k})} \cK^{p}\}$ converges to zero $\nu^{n+2}$ a.e., and is bounded by the $L^{1}$ function $\cK^{p}$. So, the dominated convergence theorem validates Claim 2. Consequently, there exists $\tau_{0}$ such that $0 < 2 \tau_{0} < 2^{-\ell}$ and
\begin{equation} \label{e:s31}
I( 2 \tau_{0}) \le  \frac{\zeta}{16C} \nu(\R^{m}).
\end{equation}

Next, define a cover of $E_{\ell}$ by
\begin{equation} \label{e:s32}
\cG = \left\{ B(x, \tau) : x \in E_{\ell}, ~ 0 < \tau < \tau_{0}, ~ \text{and} ~ \Theta^{n}(\nu,x ,\tau) \ge \frac{c}{2} \right\}.
\end{equation}
In particular, \eqref{e:smalllowerreg} guarantees $\cG$ is a fine cover of some $\nu$-measurable set $E \subset E_{\ell}$ with $\nu$-full measure, i.e., $\nu(E) = \nu(\R^{m})$. So, the corollary to Besicovitch's covering theorem, \cite[Corollary 1.5.2]{gariepy1992measure} ensures that there exists a countable, disjoint subfamily $\{B_{i}\}$ of $\cG$ with 

\begin{equation} \label{e:s33}
\nu(\R^{m} \setminus \bigcup_{i=1}^{\infty} B_{i}) = 0.
\end{equation}

In light of \eqref{e:EL} and \eqref{e:s32}, $\tau < \tau_{0} < 2^{- \ell}$ ensures
\begin{equation*}
\nu(\R^{m}) = \sum_{i=1}^{\infty} \nu(B_{i}) \le \sum_{i=1}^{\infty} (2C) \left( \frac{  \diam B_{i}}{2} \right)^{n}
\end{equation*}
so that
\begin{equation} \label{e:s34}
\frac{\nu( \R^{m})}{2C} \le \sum_{i=1}^{\infty} \left( \frac{ \diam B_{i}}{2} \right)^{n}.
\end{equation}

\noindent Moreover, $(B_{i} \cap E_{\ell})^{n+2} \subset A(2 \tau_{0}) \cap B_{i}$ so, \eqref{e:s31} and \eqref{e:Itau} yields
\begin{equation} \label{e:s35}
\sum_{i=1}^{\infty} \cM_{\cK^{p}} ( \nu \restr B_{i} ) \le I(2 \tau_{0}) \le \frac{\zeta}{16C} \nu(\R^{m}).
\end{equation}

\noindent Define the index set of ``bad" balls, or balls with too much Menger curvature by
\begin{equation} \label{e:s36}
I_{b} = \left\{ i \in \N : \cM_{\cK^{p}}( \nu \restr B_{i} ) \ge \zeta \frac{ \left( \frac{ \diam B_{i}}{2} \right)^{n}}{4} \right\}
\end{equation}

\noindent Then,
\begin{align} \label{e:s37}
\sum_{i \in I_{b}} \cM_{\cK^{p}}( \nu \restr B_{i}) \ge \frac{\zeta}{4} \sum_{i \in I_{b}} \left( \frac{ \diam B_{i}}{2} \right)^{n}
\end{align}

\noindent Notice that if $\sum_{i \in I_{b}} \left( \frac{ \diam B_{i}}{2} \right)^{n} > \frac{ \nu(\R^{m})}{4C}$ then additionally considering \eqref{e:s35} and \eqref{e:s37} implies
$$
\sum_{i \in \N} \cM_{\cK^{p}}(\nu \restr B_{i}) \le \frac{\zeta}{16C} \nu(\R^{m}) <  \frac{\zeta}{4} \sum_{i \in I_{b}} \left( \frac{ \diam B_{i}}{2} \right)^{n} \le \sum_{i \in I_{b}} \cM_{\cK^{p}}(\nu \restr B_{i})
$$
which is a contradiction. It follows 
\begin{equation} \label{e:s38}
\sum_{i \in I_{b}} \left( \frac{ \diam B_{i}}{2} \right)^{n} \le \frac{ \nu(\R^{m})}{4C}.
\end{equation} 
Now, \eqref{e:s34} and \eqref{e:s38} together ensure that $I_{b} \neq \N$.

From now on, fix $i \in \N \setminus I_{b}$. The inner-regularity of Radon measures ensures that there exists some compact $E^{*}$ with 

\begin{equation} \label{e:s310}
E^{*} \subset B_{i} \cap E_{\ell} ~ \text{ and } ~ \nu(E^{*}) \ge \frac{1}{2} \nu(B_{i}).
\end{equation} 
Then evidently, $E^{*}$ satisfies:

\begin{enumerate}
\item $ \nu(E^{*}) \ge \frac{1}{2} \nu(B_{i}) \ge \frac{1}{2} \frac{c}{2} \left( \frac{ \diam B_{i}}{2} \right)^{n} \ge \frac{c}{2^{n+2}} \diam(E^{*})^{n}$, where the second inequality is because $B_{i} \in \cG$, see \eqref{e:s32}.

\item For all $x \in E^{*}$ and for all $0 < t < 2^{-\ell}$ it follows from \eqref{e:EL} and $E^{*} \subseteq E_{\ell}$ that
\begin{equation} \label{e:s39}
\nu \left( E^{*} \cap B(x,t) \right) \le 2 C \left( \frac{ \diam \left( E^{*} \cap B(x,t) \right)}{2} \right)^{n}.
\end{equation}
On the other hand, $E^{*} \subset B_{i}$ and $\diam(B_{i}) \le 2 \tau_{0} < 2^{- \ell}$. So, for $t \ge 2^{-\ell}$ it follows,
$$
\nu \left( E^{*} \cap B(x,t) \right) \le \nu(B_{i}) \le 2C \left( \frac{ \diam B_{i}}{2} \right)^{n} < 2C \left( 2^{-\ell} \right)^{n} \le 2C t^{n}
$$
so that in fact $\nu \restr E^{*}$ is $n$-Ahlfors upper regular with regularity constant $2C$, that is \eqref{e:s39} holds for all $t > 0$.

\item It follows 
\begin{equation}\label{e:s41}
\frac{1}{4} \left( \frac{\diam(B_{i})}{2} \right)^{n}  \le \frac{2C}{c} (\diam E^{*})^{n}.
\end{equation} 
Indeed, choose a ball $B$ with 
\begin{equation} \label{e:s40}
\diam B \le 2 \diam E^{*} ~ \text{ and } ~ E^{*} \subset B.
\end{equation} 
Combining \eqref{e:s32}, \eqref{e:s310}, \eqref{e:s39}, and \eqref{e:s40} yields
$$
\frac{c}{4} \left( \frac{ \diam B_{i}}{2} \right)^{n} \le \frac{\nu(B_{i})}{2} \le \nu(E^{*})  = \nu(E^{*} \cap B) \le 2C \left( \frac{ \diam B}{2} \right)^{n} \le 2C \diam(E^{*})^{n}.
$$
Hence, the lower bound on $\diam E^{*}$ follows.

\item Finally, since $i \in  \N \setminus I_{b}$, \eqref{e:s36} and \eqref{e:s41} yields
$$
\cM_{\cK^{p}}( \nu \restr E^{*}) < \frac{ \zeta}{4} \left( \frac{ \diam B_{i}}{2} \right)^{n} \le  \frac{2C}{c} \zeta \left(  \diam E^{*} \right)^{n}
$$
\end{enumerate}
Choosing $\zeta = \frac{c \zeta_{0}}{2C}$ completes the proof.
\end{proof}

The next technical lemma is a standard ``structure theorem" and is contained in for instance \cite[3.3.12 - 3.3.15]{federer2014geometric}
\begin{lemma} \label{l:RandURparts}
If $\mu$ is a Radon measure with $0 < \Theta^{n,*}(\mu,x)$ for $\mu$ almost every $x$, then one can write $\mu = \mu_{r} + \mu_{u}$ where $\mu_{r} = \mu \restr E$ for some $\mu$-measurable $E$, and $\mu_{r}$ is countably $n$-rectifiable. On the other hand, $\mu_{u}$ is purely unrectifiable.
\end{lemma}

Now, we are prepared to show 
\begin{lemma} \label{l:54}
Let $\mu$ be a Radon measure on $\R^{m}$ with $0 < \Theta^{n,*}(\mu,x) < \infty$ for $\mu$ almost every $x \in \R^{m}$. Let $\cK^{2}$ a $(\mu,2)$-proper integrand. After writing $\mu = \mu_{u} + \mu_{r}$ as in Lemma \ref{l:RandURparts}, if $\mu_{u}(\R^{m}) > 0$ then $\cM_{\cK^{2}}(\mu) = + \infty$.
\end{lemma}

Notably, Lemma \ref{l:54} is the contrapositive of Theorem \ref{t:mainrect}.

\begin{proof}
Let $\cK^{2}$ and $\mu, \mu_{u}, \mu_{r}$ be as in the statement of Lemma \ref{l:54}. Without loss of generality, suppose $0 < \mu(\R^{m}) < \infty$. It follows $0 < \Theta^{n,*}(\mu_{u},  x) < \infty$ for $\mu_{u}$ almost every $x \in \R^{m}$, since the Lebesgue-Besicovitch differentiation theorem guarantees
$$
\Theta^{n,*}(\mu_{u},x) = \limsup_{r \to 0} \frac{ \mu_{u}(B(x,r))}{r^{n}} = \limsup_{r \to 0} \frac{ \mu_{u}(B(x,r))}{\mu(B(x,r))} \frac{ \mu(B(x,r))}{r^{n}} = \Theta^{n,*}(\mu,x)
$$
for $\mu_{u}$ a.e. $x \in \R^{m}$.

Moreover, by assumption $\mu_{u}(\R^{m}) > 0$. By Lemma \ref{l:step0} it follows that there exists $\nu_{0}$ a restriction of $\mu_{u}$ and some $c, C$ such that $0 < c \le \Theta^{n,*}(\nu_{0}, x) \le C < \infty$ for $\nu_{0}$ almost every $x \in \R^{m}$. Lemma \ref{l:step0} also guarantees that, $0 < \frac{1}{2} \mu_{u}(\R^{m}) <  \nu_{0}(\R^{m}) \le \mu(\R^{m}) < \infty$. Since $\nu_{0}$ is a restriction of $\mu_{u}$ to some Borel set, it follows that $\nu_{0}$ is a Radon measure satisfying $\cM_{\cK^{2}}(\nu_{0}) \le \cM_{\cK^{2}}(\mu_{u}) \le \cM_{\cK^{2}}(\mu)$. In the spirit of contradiction, suppose $\cM_{\cK^{2}}(\mu) < \infty$.

Since $\nu_{0}(\R^{m}) < \infty$, without loss of generality, suppose $\spt(\nu_{0})$ is bounded. Then, $\nu_{0}$ satisfies the hypothesis of Lemma \ref{l:53}. In particular, for 
\begin{equation}\label{e:etazeta}
\zeta_{0} < \eta \cdot \left( \frac{ 2^{n+2}}{c} \right)^{-(n+2)}  \text{ where }\eta = \eta \left( \cK, n, m, 2^{n+3}   C c^{-1}  \right) \text{ is from Theorem \ref{t:maingraph}},
\end{equation}
there exists some compact $E^{*} \subset \spt(\nu_{0})$ such that

\begin{enumerate}
\item[(i)] 
\begin{equation} \label{e:s14}
\nu_{0}(E^{*}) \ge \frac{c}{2^{n+2}} \left( \diam E^{*} \right)^{n}
\end{equation}
\item[(ii)] For all $x \in E^{*}$ and all $t > 0$, 
\begin{equation} \label{e:s17}
\nu_{0} \left( E^{*} \cap B(x,t) \right) \le 2 C t^{n}.
\end{equation}
\item[(iii)] 
\begin{equation} \label{e:s19}
\cM_{\cK^{p}}(\nu_{0} \restr E^{*}) \le \zeta_{0} \left( \diam E^{*} \right)^{n}
\end{equation}
\end{enumerate}

Our next goal is to scale and translate $\nu_{0}$ to find a measure $\nu_{1}$ which satisfies the hypothesis of Theorem \ref{t:maingraph}. To this end, choose $x_{0} \in E^{*}$. Then,
\begin{equation} \label{e:s311}
E^{*} \subset B(x_{0}, \diam(E^{*})).
\end{equation}
Let $f(y) = \frac{y - x_{0}}{\diam (E^{*})}$, so that 
\begin{equation} \label{e:s312} 
f^{-1}(B(x_{0}, \diam(E^{*}))) = B(0,1).
\end{equation} 

Define 

\begin{equation} \label{e:nu1}
\nu_{1} = (\diam E^{*})^{-n} \left( \frac{2^{n+2}}{c} \right) f_{\sharp} (\nu_{0} \restr E^{*}).
\end{equation} 

It follows by computations similar to those at the beginning of the proof of Corollary \ref{c:s1} that $\nu_{1}$ satisfies the hypotheses of Theorem \ref{t:maingraph}.

Consequently, Theorem \ref{t:maingraph} ensures there exists a Lipschitz graph $\Gamma$ such that
$$
\frac{\nu_{1}(\R^{m} \setminus \Gamma)}{\nu_{1}(\R^{m})} < \frac{1}{100}.
$$
But, recalling \eqref{e:nu1}, it is clear this implies
$$
\frac{ \left( \nu_{0} \restr E^{*} \right) \left( \R^{m} \setminus f(\Gamma) \right)}{\nu_{0}(\R^{m})} < \frac{1}{100}.
$$
Since $f$ is a translation and scaling, $f(\Gamma)$ is still a Lipschitz graph, contradicting the fact that $\nu_{0}$ is the restriction of an $n$-purely unrectifiable measure.
\end{proof}

\subsection{Pointwise Menger curvature and $\beta$-numbers.} 

The first goal of this section is to prove Theorem \ref{t:maincharacterization}, included below for convenience.
\begin{theorem} \label{t:mainCharacterization}
Let $\mu$ be a Radon measure on $\R^{m}$ with $0 < \Theta^{n}_{*}(\mu,x) \le \Theta^{n,*}(\mu,x) < \infty$ for $\mu$ almost every $x \in \R^{m}$. Then the following are equivalent:
\begin{enumerate}
\item $\mu$ is countably $n$-rectifiable.
\item For $\mu$ almost every $x \in \R^{m}$, $\curv^{n}_{\cK_{2};\mu}(x,1) < + \infty$.
\item For $\mu$ almost every $x \in \R^{m}$, $\curv^{n}_{\cK_{1};\mu}(x,1) < + \infty$.
\item $\mu$ has $\sigma$-finite integral Menger curvature in the sense that $\mu$ can be written as $\mu = \sum_{j=1}^{\infty} \mu_{j}$ where each $\mu_{j}$ satisfies
$\displaystyle
\cM_{\cK_{1}^{2}}(\mu_{j}) < \infty.
$
\item $\mu$ has $\sigma$-finite integral Menger curvature in the sense that $\mu$ can be written as $\mu = \sum_{j=1}^{\infty} \mu_{j}$ where each $\mu_{j}$ satisfies
$\displaystyle
\cM_{\cK_{2}^{2}}(\mu_{j}) < \infty.
$

\end{enumerate}
\end{theorem}

\begin{proof}
The fact that (1) $\implies$ (2) is the content of \cite[Lemma 1.1]{kolasinski2016estimating} combined with the characterization by Azzam and Tolsa in Theorem \ref{t:azzchar}. Since $\cK_{1} \le \cK_{2}$ pointwise, it also follows that (2) $\implies$ (3). So, it suffices to show $(3) \implies (4) \implies (1)$ and $(2) \implies (5) \implies (1)$.

To this end, fix $\mu$ as in the statement of the theorem. Moreover, without loss of generality suppose that 
\begin{equation} \label{e:331}
\mu(\R^{m}) < \infty
\end{equation} 

Then, for $j \in \N_{0}$ define
$$
E_{j} = \{ x \in \R^{m} : \curv^{n}_{\cK_{1}; \mu}(x,1) \in [j, j+1) \} \quad \text{and} \quad \mu_{j} = \mu \restr E_{j}.
$$
Fubini's theorem \cite[Theorem 1.4.1]{gariepy1992measure} ensures that the map 
$$
x \mapsto \int_{(\R^{m})^{n+1}} \cK^{2}(x,x_{1}, \dots, x_{n+1}) \dif \mu^{n+1}(x_{1}, \dots, x_{n+1}) = \curv^{n}_{\cK_{1};\mu}(x,\infty)
$$ 
is measurable. In particular the sets $E_{j}$ are measurable, as each $E_{j}$ is the preimage of a Borel set by a measurable function. For each $j  \in \N_{0}$ it follows from \eqref{e:331} that $\mu_{j} = \mu \restr E_{j}$ is a finite, Borel measure. The Lebesgue-Besicovitch differentiation theorem ensures that $0 < \Theta^{n,*}(\mu_{j},x) < \infty$ for $\mu_{j}$ a.e. $x \in \R^{m}$ since $\displaystyle \lim_{r \to 0} \frac{ \mu_{j}(B(x,r))}{\mu(B(x,r))} = 1$ for $\mu_{j}$ a.e. $x \in \R^{m}$. Since $\cK^{2}$ is a non-negative function, and $\mu_{j}$ satisfies $\mu_{j}(E) \le \mu(E)$ for all $\mu$-measurable sets $E$ it follows
\begin{align*}
\curv^{n}_{\cK_{1};\mu_{j}}(x, \infty) &= \int_{(\R^{m})^{n+1}} \cK^{2}(x,x_{1}, \dots, x_{n+1}) \dif \mu_{j}^{n+1}(x_{1}, \dots, x_{n+1}) \\
& \le \int_{(R^{m})^{n+1}}  \cK^{2}(x, x_{1},\dots, x_{n+1}) \dif \mu^{n+1}(x_{1}, \dots, x_{n+1}) \\
& \le (j+1)  \qquad \forall x \in E_{j}.
\end{align*}
Combining the above computation with \eqref{e:331} yields,
$$
\cM_{\cK_{1}^{2}}(\mu_{j}) = \int_{\R^{m}} \curv^{n}_{\cK_{1};\mu_{j}}(x,\infty) \dif \mu_{j}(x) < (j+1) \mu_{j}(\R^{m}) < \infty.
$$
Therefore, $(3) \implies (4)$. To see $(4) \implies (1)$ , note that Theorem \ref{t:mainrect} ensures each $\mu_{j}$ is $n$-countably rectifiable. Since $\mu = \sum_{j} \mu_{j}$ satisfies $\mu \ll \cH^{n}$ has been decomposed into countably many $n$-countably rectifiable pieces, it follows $\mu$ is countably $n$-rectifiable.

The fact that $(2) \implies (5) \implies (1)$ follows identically with $\cK_{2}$ in place of $\cK_{1}$.
\end{proof}

Given the history of the subject, this method of proof is not very satisfying, namely due to the fact that Lerman and Whitehouse's characterization of uniform rectifiability in terms of integral Menger curvature demonstrated an equivalence of a Carleson-type condition for integral Menger curvature and the Carleson condition for the $\beta$-numbers. Moreover, the first direction in this characterization follows from a direct comparison with $\beta$-numbers due to Kolasi{\'n}ski.

Our next goal is to show Theorem \ref{t:mainbeta}, which is equivalent to Theorem \ref{t:mainBeta}. We simply use Lemma \ref{l:generalityofmain} to restate the theorem for the sake of illuminating the dependencies on $x$.

\begin{theorem} \label{t:mainBeta}
If $\mu$ is an $n$-Ahlfors upper-regular Radon measure on $\R^{m}$ with upper-regularity constant $C_{0}$, and there exists $\lambda$ such that $\mu(B(x,r)) \ge \lambda r^{n}$ for all $0 < r \le R$, then
\begin{equation} \label{e:58}
\int_{0}^{R} \obeta^{n}_{\mu;2}(x,r)^{2} \frac{ \dif r}{r} \le C_{3} \curv^{n}_{\mu;\cK^{2}}(x,R),
\end{equation}
where $C_{3} = C_{3}(\lambda,m,n, C_{0}, \cK)$, and $\cK$ is any $(\mu,2)$-proper integrand.

In fact, if $\mu, \lambda, x,r$, and $R$ are as above, then for $\cK \in \{\cK_{1}, \cK_{2}\}$ 
\begin{align} \label{e:59}
\int_{0}^{R} \obeta_{\mu;2}^{n}(x,r)^{2} \frac{ \dif r}{r} &\le C \curv^{n}_{\mu;\cK}(x,R) \le  C \cdot \Gamma \int_{0}^{2R} \Theta^{n}(\mu,x,r) \obeta_{\mu;2}^{n}(x,r)^{2} \frac{ \dif r}{r} \\
\tag*{} &\le \tilde{C} \int_{0}^{2R} \obeta_{\mu;2}^{n}(x,r)^{2} \frac{ \dif r}{r}
\end{align}
with constants $C, \Gamma, \tilde{C}$ depending on $m,n, \lambda$ the upper-regularity constant of $\mu$, and $\cK$.
\end{theorem}

This theorem demonstrates a more direct converse to Kolasi{\'n}ski's bound on $\beta$-numbers. Alas, notice that in its present form, Theorem \ref{t:mainBeta} requires much stronger density conditions than Theorem \ref{t:mainCharacterization}. It would be interesting to try to weaken the density conditions at least as far as they are in Theorem \ref{t:mainCharacterization}.

The following technical lemma plays a central role in the proof of Theorem \ref{t:mainBeta}. For a review of the notation used in the proof, see Section \ref{s:notation}.

\begin{lemma} \label{t:likeNV}
Let $\mu$ be an $n$-Ahlfors upper-regular Radon measure on $\R^{m}$ with upper-regularity constant $C_{0}$. Suppose $x \in \R^{m}$ and $\lambda, R > 0$ such that 
\begin{equation} \label{E:1}
\mu(B(x,r)) \ge \lambda r^{n}
\end{equation} 
holds for all $0 < r \le R$.

Then, for
\begin{equation} \label{E:deltadef}
\delta = \delta(n, \lambda, C_{0}) = \frac{\lambda}{2^{k+2} 5^{n-1} C_{0}} 
\end{equation}
and 
\begin{equation} \label{E:etadef}
\eta = \eta(n, \lambda, C_{0}) = \frac{\delta}{10n} = \frac{\lambda}{2^{k+3} 5^{n} n C_{0}}
\end{equation}
and all $0 < r \le R$ there exist points $\{x_{i,r}\}_{i=1}^{n}  \subset B(x,r)$ such that
\begin{equation} \label{E:bigh}
h_{\min}(x, x_{1,r}, \dots, x_{n,r}) \ge \delta r
\end{equation}
and
\begin{equation} \label{E:0}
(\mu \restr B(x,r)) (B(x_{i,r}, 5 \eta r)) \ge  \left( \frac{\lambda \eta^{m}}{2^{m+1}} \right) r^{n} = C_{2}(m,n,\lambda, C_{0}) r^{n}.
\end{equation}
In particular, if for each $i \in \{1, \dots, n\}$, $B_{i,r} \defeq B(x_{i,r}, 5 \eta r)$ for any choices of $y_{i} \in B_{i,r}$ it follows that
\begin{equation}\label{E:fczfat}
h_{\min}(x,y_{1}, \dots, y_{n}) \ge \delta r - 5 n \eta r = \frac{\delta r}{2}.
\end{equation} 

Finally, if $\mathbb{B}_{r} \defeq B_{1,r} \times \dots \times B_{n,r}$ then 
\begin{equation} 
\label{e:disjoint} \mathbb{B}_{\delta r/3} \cap \mathbb{B}_{r} = \emptyset.
\end{equation}
\end{lemma}

\begin{proof} (of Lemma \ref{t:likeNV}).
Let $m, n, \mu, C_{0}, \lambda$ and $R$ be as in the theorem statement. Define $\delta, \eta$ as in \eqref{E:deltadef} and \eqref{E:etadef}. 


Fix $0 < r < R$, and suppose there exist $\{x_{1}, \dots, x_{k}\}$ satisfying 
$$
h_{\min}(x,x_{1}, \dots, x_{k}) \ge \delta r \text{ and } \mu(B(x_{i}, 5 \eta r)) \ge \left( \frac{ \lambda \eta^{m}}{2^{m+1}} \right) r^{n}
$$ 
for all $i = 1, \dots, k$ and assume that $k < n$. Then, we will find a point $x_{k+1}$ such that $h_{\min}(x,x_{1}, \dots, x_{k+1}) \ge \delta r$ and $\mu(B(x_{k+1}, 5 \eta r)) \ge \left( \frac{ \lambda \eta^{m}}{2^{m+1}} \right) r^{n}$. Hence, induction will guarantee the theorem.\footnote{The proof of the inductive step clearly shows that we can also find a point $x_{1}$ with $|x_{1} - x| \ge \delta r$ and $\mu(B(x_{1}, 5 \eta r)) \ge \left( \frac{ \lambda \eta^{m}}{2^{m+1}} \right) r^{n}$.}

Let $V_{k} = \aff\{x, x_{1}, \dots, x_{k}\} = x + \Span \{x_{1} - x, \dots, x_{k} - x\}$.  Define
\begin{equation} \label{e:vkdr}
(V_{k})_{\delta r} = \{ y \in B(x,r) : \dist(y, V_{k}) < \delta r\}.
\end{equation}

Define $\rho = \rho(\lambda, C_{0}, n,k,r)$ by 
\begin{equation} \label{e:rhodef}
\rho = s r \quad \text{ where } s = \left( \frac{ \lambda }{C_{0}}  \frac{1}{2^{k+1} \cdot 5^{n} } \right)
\end{equation}

Let
\begin{equation} \label{e:CG1}
\cG_{1} = \{ B(y, 5 \rho) \mid y \in V_{k} \cap B(x,r) \}.
\end{equation}
We first note that $\cG_{1}$ is a cover of $B(x,r) \cap (V_{k})_{\delta r}$ since $\delta = \frac{5}{2}s$ implies $\delta r = \frac{5 \rho}{2}$. So, by Vitali we can find a subfamily of sets $\{B(x_{i}, 5 \rho) \}_{i=1}^{N^{\prime}}$ such that $B(x,r) \cap (V_{k})_{\delta r} \subset \bigcup_{i} B(x_{i}, 5 \rho)$  and $\{B(x_{i}, \rho)\}_{i=1}^{N^{\prime}}$ is disjoint.

A priori, $N^{\prime}$ could be infinite, but we will see that $N^{\prime} \le N_{k,s} = N_{n,k,\lambda, C_{0}}$ where
\begin{equation} \label{e:Nks}
N_{n,k,\lambda, C_{0}} = \left( \frac{2}{s} \right)^{k}.
\end{equation}

Indeed, since $B \in \cG_{1} \implies B \cap V_{k}$ is a $k$-dimensional ball of radius $5 \rho$, 
\begin{align*}
N^{\prime} \omega_{k} \rho^{k} &= \sum_{i=1}^{N^{\prime}} \cH^{k} \left( V_{k} \cap B\left(x_{i} , \rho \right) \right) = \left(\cH^{k}  \restr V_{k} \right) \left( \bigcup_{i=1}^{N^{\prime}} B\left(x_{i}, \rho \right) \right) \\
& \le \left(\cH^{k} \restr V_{k} \right) \left( B(x, 2r) \right) = \omega_{k} (2r)^{k}
\end{align*}
so that $N^{\prime} \le 2^{k} (r \rho^{-1})^{k} = (2 s^{-1})^{k} = N_{n,k,\lambda, C_{0}}$.

We wish to show that our choice of $\delta$ forces $\displaystyle \mu(B(x,r) \cap (V_{k})_{\delta r}) \le \frac{ \lambda r^{n}}{2}$ so that 
\begin{equation} \label{E:bigcomp}
\displaystyle \mu(B(x,r) \setminus (V_{k})_{\delta r}) \ge \frac{ \lambda r^{n}}{2}.
\end{equation}

Indeed,
\begin{align*}
\mu \left( (V_{k})_{\delta r} \cap B(x,r) \right) &\le \sum_{i=1}^{N^{\prime}} \mu(B(x_{i}, 5 \rho)) \le \sum_{i=1}^{N^{\prime}} C_{0} (5 \rho)^{n} \le N_{n,k, \lambda, C_{0}} C_{0} 5^{n} s^{n} r^{n} \\
& = \left( \frac{2}{s} \right)^{k} C_{0} 5^{n} s^{n} r^{n}.
\end{align*}
So, it suffices to show
$$
2^{k} C_{0} 5^{n} s^{n-k} r^{n} \le \frac{ \lambda r^{n}}{2}
$$
which holds if and only if
$$
s^{n-k} \le \frac{ \lambda }{2^{k+1}  5^{n}  C_{0}}.
$$
Since $k < n$ this implies that our choice of $s$ in \eqref{e:rhodef} suffices to ensures \eqref{E:bigcomp}.

Now we claim that \eqref{E:bigcomp} guarantees the existence of some $x_{k+1} \in B(x,r) \setminus (V_{k})_{\delta r}$ such that \eqref{E:0} holds. The fact that $x_{k+1} \notin (V_{k})_{\delta r}$ will guarantee \eqref{E:bigh}.

To this end, let us consider the family of balls
$$
\cG_{2} = \{ B(y, 5 \eta r) \mid y \in B(x,r) \setminus (V_{k})_{\delta r} \}.
$$
Then $B \in \cG_{2}$ implies $B \subset B(x,2r)$ since $5 \eta < \frac{\delta}{2n} < 1$. Moreover, $\cG_{2}$ covers $B(x,r) \setminus (V_{k})_{\delta r}$. In particular, Vitali ensures there exists a subfamily $\{B(x_{i}, 5 \eta r) \}_{i=1}^{M^{\prime}}$ that covers $B(x,r) \setminus (V_{k})_{\delta r}$ and $\{B(x_{i}, \eta r) \}_{i=1}^{M^{\prime}}$ is a disjoint family. Again, we have no apriori estimate on $M^{\prime}$, but disjointness and containment in $B(x,2r)$ yields
$$
\omega_{m}  (r \eta)^{m} M^{\prime} = \sum_{i=1}^{M^{\prime}} \cH^{m} \left( B(x_{i}, r \eta )\right) \le \cH^{m}(B(x,2r)) = \omega_{m} (2r)^{m}.
$$
Consequently, we define $M_{m,n,\lambda,C_{0}}$ so that
\begin{equation} \label{E:mprime}
M^{\prime} \le (2 \eta^{-1})^{m} = M_{\eta, m} = M_{m,n,\lambda, C_{0}}.
\end{equation}

Combining \eqref{E:bigcomp} and \eqref{E:mprime}, we deduce
\begin{align} 
\nonumber \frac{ \lambda r^{n}}{2} &\le \mu \left( B(x,r) \setminus (V_{k})_{\delta r} \right) \\
\nonumber & \le \left( \mu \restr B(x,r) \right) \left( \bigcup_{i=1}^{M^{\prime}} B(x_{i}, 5 \eta r) \right) \\ 
\nonumber & \le \sum_{i=1}^{M^{\prime}} (\mu \restr B(x,r)) \left( B(x_{i}, 5 \eta r) \right) \\
\label{E:10} &\le M_{m,n,\lambda, C_{0}} \max \left \{ (\mu \restr B(x,r)) (B(x_{i}, 5 \eta r) ) \mid i \in \{1, \dots, M^{\prime} \} \right\}.
\end{align}
Choosing $k+1 = j$ such that 
$$
(\mu \restr B(x,r))(B(x_{j}, 5 \eta r)) = \max \left \{ (\mu \restr B(x,r)) (B(x_{i}, 5 \eta r) ) \mid i \in \{1, \dots, M^{\prime} \} \right\},
$$
we have from \eqref{E:10} that
$$
(\mu \restr B(x,r)) (B(x_{k+1}, 5 \eta r)) \ge \frac{ \lambda r^{n}}{2 M_{m,n,\lambda, C_{0}}} \ge \frac{ \lambda r^{n} \eta^{m}}{2^{m+1}}
$$
verifying that $x_{k+1} \in B(x,r) \setminus (V_{k})_{\delta r}$ satisfies \eqref{E:0}.

It only remains to show \eqref{E:fczfat} and \eqref{e:disjoint}, which follow quickly from the work already done. Indeed, Lemma \ref{l:boundedheightchange} and \eqref{E:bigh} verify \eqref{E:fczfat}. On the other hand, \eqref{e:disjoint} follows from $\mathbb{B}_{\delta r/3} \subset B(x, \delta r/3)$ and \eqref{E:bigh}. 
\end{proof}

\begin{proof}(Of Theorem \ref{t:mainBeta})
Fix $\mu, R, \lambda $ as in the theorem statement. Let $\cK$ be some $(\mu,2)$-proper integrand, and $0 < r \le R$. Let $\{x_{i,r}\}$, $B_{i,r}$ and $\mathbb{B}_{r}$ be as in Lemma \ref{t:likeNV}. Then, first replacing the infimum with an average over fixed planes, and then applying \eqref{E:0}, yields
\begin{align*}
 \obeta^{n}_{\mu;2}&(x,r)^{2}  = \inf_{L \ni x} \frac{1}{r^{n}} \int_{B(x,r)} \left( \frac{\dist(z, L)}{r} \right)^{2} d \mu(z) \\
& \le \int_{\mathbb{B}_{r}} \int_{B(x,r)}  \left(\frac{\dist(z,\aff\{x,y_{1}, \dots, y_{n}\})^{2}}{r}\right)^{2} \frac{ d\mu(z) d \mu^{n}(y_{1}, \dots, y_{n})}{\mu^{n}(\mathbb{B}_{r}) r^{n}} \\
& \le C \int_{\mathbb{B}_{r}} \int_{B(x,r)} \left(\frac{ \dist(z, \aff \{x,y_{1}, \dots, y_{n}\})}{r} \right)^{2} \frac{d \mu(z) d \mu^{n}(y_{1}, \dots, y_{n})}{r^{n^{2} + n}},
\end{align*}
where $C= C(m,n,\lambda, C_{0})$. Since $\{x,z\} \cup B_{i,r} \subset B(x,r)$ for all $i=1, \dots, n$,  \eqref{E:fczfat} ensures we can apply \eqref{e:p303} in the final integral above, so that
\begin{equation} \label{e:betaest1}
\obeta^{n}_{\mu;2}(x,r)^{2} \le C \int_{\mathbb{B}_{r}} \int_{B(x,r)} \cK(x,z,y_{1}, \dots, y_{n})^{2} d\mu(z)d \mu^{n}(y_{1}, \dots, y_{n}).
\end{equation}
Finally, using the fact that for any $0 < \sigma < 1$, 
\begin{equation*}
\int_{0}^{R} \obeta^{n}_{\mu;2}(x,r)^{2} \frac{dr}{r} \le C_{\sigma} \sum_{j \ge 0} \obeta^{n}_{\mu;2}(x, \sigma^{j} R)
\end{equation*}
when $\sigma = \delta/3$ and writing $r_{j} \defeq \left( \frac{\delta}{3} \right)^{j} R$, \eqref{e:disjoint} and \eqref{e:betaest1} yield
\begin{equation*}
\int_{0}^{R} \obeta^{n}_{\mu;2}(x,r)^{2} \frac{dr}{r} \le C \int_{\cup_{j \ge 0} \mathbb{B}_{r_{j}}} \int_{B(x,r)} \cK(x,z,y_{1}, \dots, y_{n})^{2} d \mu^{n+1}(z,y_{1}, \dots, y_{n}),
\end{equation*}
where $C=C(m,n,\lambda, C_{0})$. Since for all $j$, $\mathbb{B}_{r_{j}} \times B(x,r) \subset B(x,r)^{n+1}$, the theorem follows from non-negativity of the integrand by replacing $\cup_{j \ge 0} \mathbb{B}_{r_{j}} \times B(x,r)$ with $B(x,r)^{n+1}$.
\end{proof}

\pagestyle{fancy}

\bibliographystyle{alpha}
\bibdata{references}
\bibliography{references}

\end{document}